\documentclass[final,onefignum,onetabnum]{siamart220329}
\usepackage{braket,amsfonts}

\usepackage{amsmath}
\usepackage{amssymb}
\usepackage{graphicx}
\usepackage[caption=false]{subfig}
\usepackage{bm}
\usepackage{epstopdf}
\usepackage{xcolor}
\usepackage{tikz}
\usepackage{multirow}
\newsiamremark{remark}{Remark}
\newsiamremark{assumption}{Assumption}

\usepackage[most]{tcolorbox}

\begin{tcbverbatimwrite}{tmp_\jobname_header.tex}
\title{A COUPLING METHOD OF MIXED AND LAGRANGE FINITE ELEMENTS FOR LINEAR ELASTICITY PROBLEM
\thanks{
Submitted to the editors xx xx, xxxx. 
\funding{The first author was  supported by the China Postdoctoral Science Foundation, Grants No. 2025M783065. The second author was supported by the National Natural Science Foundation of China, Grants No. NSFC 12288101. The third author was supported by the National Natural Science Foundation of China, Grants No. NSFC 12301523. The last author was  supported by the China Postdoctoral Science Foundation, Grants No. 2025M773114}}
}

\author{Wei Chen\thanks{LMAM and School of Mathematical Sciences, Peking University, Beijing 100871, P. R. China~\&~Chongqing Research Institute of Big Data, Peking University, Chongqing 401329, P. R. China (\email{2406397052@pku.edu.cn},\email{myzhang@csrc.ac.cn}).}
\and Jun Hu\thanks{LMAM and School of Mathematical Sciences, Peking University, Beijing 100871, P. R. China (\email{hujun@math.pku.edu.cn}).}
\and Limin Ma\thanks{School of Mathematics and Statistics, Wuhan University, Wuhan 430072, Hubei, P. R. China~\&~National Center for Applied Mathematics in Hubei, Wuhan 430072, Hubei, P. R. China (\email{limin18@whu.edu.cn}).}
\and Mingyan Zhang\footnotemark[2]
}

\headers{A COUPLING METHOD OF MIXED AND LAGRANGE ELEMENTS}{Wei Chen, Jun Hu, Limin Ma, and Minyan Zhang}
\end{tcbverbatimwrite}
\input{tmp_\jobname_header.tex}

\ifpdf
\hypersetup{pdftitle={Guide to Using  SIAM'S \LaTeX\ Style} }
\fi

\newcommand{\Omo}{{\Omega^+_h}}
\newcommand{\Omt}{{\Omega^-_h}}
\newcommand{\Omi}{{\Omega^\pm_h}}
\newcommand{\bR}{\mathbb{R}}
\newcommand{\Rn}{\mathbb{R}^n}
\newcommand{\Rnn}{\mathbb{R}^{n \times n}}
\newcommand{\Snn}{\mathbb{S}}
\newcommand{\Div}{\operatorname{div}}
\newcommand{\Tr}{\operatorname{tr}\,}
\newcommand{\Al}{\mathcal{A}}
\newcommand{\Gai}{\Gamma^\pm_h}
\newcommand{\Gao}{\Gamma^+_h}

\newcommand{\Ga}{{\Gamma_h}}
\newcommand{\FhGa}{\mathcal{F}_{\Gamma,h}}
\newcommand{\ToGa}{\mathcal{T}_{\Gamma,h}^+}
\newcommand{\TtGa}{\mathcal{T}_{\Gamma,h}^-}

\newcommand{\uo}{u^+}
\newcommand{\vo}{v^+}
\newcommand{\vt}{v^-}
\newcommand{\ut}{u^-}
\newcommand{\uoh}{u^+_h}
\newcommand{\voh}{v^+_h}
\newcommand{\vth}{v^-_h}
\newcommand{\uth}{u^-_h}
\newcommand{\uthx}{u^*_h}
\newcommand{\Vo}{V^+}
\newcommand{\sigo}{\sigma^+}
\newcommand{\sigt}{\sigma^-}

\newcommand{\sigth}{\sigma^-_h}

\newcommand{\no}{n^+}
\newcommand{\nt}{n^-}
\newcommand{\taut}{\tau^-}
\newcommand{\tauth}{\tau^-_h}
\newcommand{\Sigt}{\Sigma^-}
\newcommand{\Vt}{V^-}
\newcommand{\Zh}{Z_h}

\newcommand{\svt}{\,\!|\!|\!|}
\newcommand{\To}{\mathcal{T}_h^+}
\newcommand{\Tt}{\mathcal{T}_h^-}
\newcommand{\Ko}{K^+}

\newcommand{\Ft}{\mathcal{F}_h^-}

\newcommand{\Voh}{V_{h}^+}
\newcommand{\Vth}{V_{h}^-}

\newcommand{\Vthx}{V_{h}^*}
\newcommand{\Sigth}{\Sigma_{h}^-}

\newcommand{\bP}{\mathbb{P}}
\newcommand{\Pits}{\Pi_{h}^-}
\newcommand{\Pth}{Q_h^-}
\newcommand{\Ioh}{I_h^+}

\newcommand{\Oms}{\Omega_S}
\newcommand{\Omr}{\Omega_R}

\newcommand{\bx}{\mbox{x}}

\newcommand{\HZ}{\mathrm{HZ}}
\newcommand{\CG}{\mathrm{L}}
\begin{document}
\maketitle

%% ------------------------------------------------------------------
%% ABSTRACT
%% ------------------------------------------------------------------
\begin{tcbverbatimwrite}{tmp_\jobname_abstract.tex}
\begin{abstract}
This paper proposes a finite element method that couples mixed and Lagrange finite elements to efficiently capture stress concentrations in elasticity problems. The method employs conforming mixed finite elements in regions with stress concentration, while standard Lagrange elements are used elsewhere, achieving a balance between stress accuracy and computational efficiency.
The well-posedness of the coupled formulation and optimal a priori error estimates are established, even when the size of the mixed finite element subregion is $O(h)$. Numerical experiments are presented to verify the theoretical convergence rates and to demonstrate the effectiveness and efficiency of the proposed method.
\end{abstract}


\end{tcbverbatimwrite}
\input{tmp_\jobname_abstract.tex}

\section{Introduction}
\label{intro}
Consider the linear elasticity equations with homogeneous boundary condition
\begin{equation}
\label{intro::linearEquation}
\begin{cases}
    \Al\sigma-\varepsilon(u)=0&\mbox{ in }\Omega,\\
    -\Div\sigma=f&\mbox{ in }\Omega,\\
%\sigma=2\mu\varepsilon(u)+\lambda\Tr\varepsilon(u)&\mbox{ in }\Omega,\\
    %A\sigma = \varepsilon(u),&\mbox{ in }\Omega,\\
    u=0&\mbox{ on }\partial\Omega,
\end{cases}    
\end{equation}
on a bounded, polygonal domain $\Omega \subset \Rn (n=2,3)$ with Lipschitz boundary~$\partial\Omega$, where~$u$ denotes the displacement, $\varepsilon(u)=\frac{1}{2}(\nabla u + \nabla u^T)$ the strain tensor, $\sigma$ the stress tensor, $f\in L^2(\Omega;\Rn)$ the body force, and~$\Snn\subset\Rnn$ the space of symmetric tensors. The compliance tensor~$\Al:\Snn\to\Snn$, characterizing the properties of the elastic material, is bounded, symmetric positive definite, i.e., there exist two positive constant $\alpha$ and $\beta$, such that
\begin{equation*}%\label{alpha-beta}
 \alpha \|\tau\|^2\leq (\Al\tau,\tau)\leq \beta\|\tau\|^2,~\text{ for all }\,\tau\in\Snn.   
\end{equation*}
For homogeneous and isotropic materials, the compliance tensor is given by
\begin{equation*}%\label{material}
\Al\tau = \frac{1}{2\mu}\left(\tau-\frac{\lambda \Tr\tau}{2\mu+n\lambda}I\right),
\end{equation*}
with the Lam\'e constants $\mu>0$ and $\lambda\geq0$, where $I$ is the identity matrix in~$\Rnn$.
% {\color{blue}
% In continuum mechanics, material properties are commonly characterized by Young’s modulus $E$ and Poisson's ratio $\nu$. They relate to the Lam\'e constants through the constitutive equations \cite[(1.2.44)]{brezzi2012mixed}:
% $$
% \mu=\frac{E}{2(1+\nu)},\quad \lambda=\frac{E\nu}{(1+\nu)(1-2\nu)}.
% $$}
In the case of plane strain in two-dimensional space and in three-dimensional space, $\Al\tau$ can be rewritten as~$\Al\tau=\frac{(1+\nu)}{E}\big(\tau-\frac{\nu\Tr\tau}{1+(n-2)\nu}I\big)$, where $E$ is Young’s modulus and $\nu$ is Poisson's ratio.
The linear elasticity model plays a significant role in solid mechanics \cite{achenbach2012wave,MR1392473,landau2012theory}. Among various weak formulations  for linear elasticity, the primal formulation and the mixed formulation are most commonly used in literature and application~\cite{MR3097958,MR1115205,MR520174,ciarlet2021mathematical}.%{\color{red}~\cite{MR3097958,brezzi2012mixed,ciarlet2021mathematical}}.

Finite element methods based on the primal formulation of linear elasticity approximate the displacement, with stresses recovered by differentiation of the discrete displacement. This typically leads to reduced accuracy in stress approximation~\cite{Brenner1992LinearFE,courant1994variational,falk1991nonconforming}. %{\color{red}~\cite{Brenner1992LinearFE,ciarlet2002finite,courant1994variational,falk1991nonconforming,hrennikoff1941solution}}. 
Lagrange finite element methods are the simplest and most widely used approaches in this framework; however, some low-order schemes may suffer from volumetric ``locking" for nearly incompressible materials \cite{ambroziak2013locking,doll2000volumetric}. 
Mixed finite element methods based on the Hellinger--Reissner formulation treat stress and displacement as independent variables and effectively avoid locking phenomena \cite{MR3097958}. Compared with standard Lagrange finite element method, this approach not only naturally enforces traction continuity across element interfaces but also preserves local conservation laws by satisfying the equilibrium equations exactly. This ensures high-quality stress approximations, especially in regions with stress concentration \cite{olesen2018a}. However, the symmetry constraint on the stress tensor makes the construction of stable conforming mixed elements highly nontrivial; see \cite{MR553347,MR761879,MR2336264,MR1930384,MR2629995,MR2831058,MR2377256} and the references therein. %{\color{red}\cite{MR553347,MR761879,MR2336264,MR1930384,MR1977627,MR2629995,MR2831058,MR2377256}}.
Recently, Hu and Zhang proposed a family of intrinsic mixed finite elements on simplicial meshes in two and three dimensions~\cite{hu2015familyconformingmixedfinite, MR3301063}, later extended to arbitrary dimensions in~\cite{MR3352360}. These elements allow direct stress approximation and are relatively simple to implement compared with earlier mixed finite elements. Furthermore, higher-order Hu--Zhang elements can be naturally employed to further improve stress accuracy.
In addition, discontinuous Galerkin (DG for short hereinafter) methods based both primal and mixed formulations are also widely used for elasticity problems due to their flexibility on non-conforming meshes, including hybridizable DG, local DG and weak Galerkin methods~\cite{soon2009hybridizable,fu2015analysis,qiu2018hdg,chen2016robust,wang2020mixed}.%{\color{red}~\cite{chen2010local,MixLDGWu,soon2008hybridizable,soon2009hybridizable,fu2015analysis,qiu2018hdg,chen2016robust,wang2016locking,wang2020mixed,wang2018hybridized,yi2019lowest}}.   

Coupling finite element methods combine different discretizations on different subdomains, enabling the use of complementary advantages of each method. Typical examples include the coupling of Lagrange and DG methods~\cite{perugia2001coupling,paipuri2019coupling,MR4056294} and mortar methods~\cite{belgacem1999mortar,MR2574903,wohlmuth2000mortar}. %{\color{red}\cite{belgacem1999mortar,bernardi1993domain,MR2574903,wohlmuth2000mortar}}. 
In most cases, interface conditions are enforced weakly via Nitsche’s method or Lagrange multipliers. For scalar elliptic problems, a coupling of Raviart--Thomas and Lagrange elements was proposed in~\cite{wieners1998coupling}, where interface conditions are incorporated directly into the formulation without using Nitsche's trick or Lagrange multipliers.

To accurately and efficiently resolve localized stress concentrations, we propose a coupling method of HuZhang mixed elements and Lagrange elements, where the Hu--Zhang elements are employed on a flexibly chosen subdomain containing stress concentration, and the Lagrange elements are used elsewhere. The interface between the two subdomains is naturally formed by the intersection of their boundaries, and can be defined as a union of $(n-1)$-dimensional faces, which is aligned with the mesh. The continuity of displacement and of the normal component of the stress tensor is imposed on the interface. Following \cite{hu2024direct},  these interface conditions are embedded directly into the weak formulation. 
For the coupling method of the Hu--Zhang element of order $k$ $(k\ge n+1$) and the Lagrange element of order $k + 1$, we establish well-posedness and optimal convergence in the energy norm, as well as optimal $L^2$ convergence via duality argument.  A superconvergence result for the displacement on mixed subdomain is proved, enabling improved approximations through postprocessing. Notably, under suitable assumptions, the analysis remains valid even when the mixed subdomain consists of only a few layers of elements surrounding the stress concentration, providing an effective balance between accuracy and computational cost. To the best of our knowledge, this is the first coupling method for linear elasticity that combines conforming mixed finite elements with standard Lagrange elements in a locally selective manner while preserving optimal convergence even when the mixed region is of size $O(h)$.
Although the analysis is restricted to displacement boundary conditions, numerical experiments for benchmark problems demonstrate the robustness and efficiency of the method for both displacement and traction boundary conditions.

%To ensure the robustness on the subdomain with stress concentration and numerical efficiency on the remaing domain, we propose a coupling method of Hu--Zhang finite elements and Lagrange finite elements for linear elasticity problems.{\color{blue} In this framework, Hu–Zhang finite elements are applied in a mesh-dependent subdomain consisting of cells in regions of stress concentration, while Lagrange finite elements are used in the remaining part of the domain. In this coupling finite element method, the interface terms in the weak formulation are treated by exchanging data across the interface: the interfacial data from one subdomain serve as natural boundary conditions at the interface for the other, and vice versa.} This coupling finite element method treats the interface condition as natural boundary conditions with respect to both subdomains, which is embedded in the weak formulation. For the coupling method of the Hu--Zhang element of order $k$~($k\ge n+1$) and Lagrange element of order $k+1$, the wellposedness and optimal convergence in energy norm are analyzed, and also the optimal convergence rate in  $L^2$ norm  by duality argument. A superconvergence result of the displacement on stress concentrate subdomain is proved, which leads to a supercongent approximate displacement by some postprocessing technique. Various numerical tests are conducted on both model problems and also  benchmark problems to demonstrate the advantages of the proposed method in terms of accuracy and robustness. 

The remainder of this paper is organized as follows.  Section \ref{notations} introduces some notations. Section \ref{con-mix} introduces and analyzes the coupling finite element method of Hu--Zhang elements and Lagrange elements.  
Section \ref{Sec:Num} presents numerical examples  to validate the advantages of the proposed coupling method.

\section{Notations and preliminaries}
\label{notations}

For a nonnegative integer $k$ and a bounded region $G\subset\Rn$, let $H^k(G; X)$ and $H(\Div,G;X)$ be the usual Sobolev space of functions taking values in the finite-dimensional vector space $X$, where $X$ can be $\Snn$, $\Rn$, or $\bR$, and $\|\cdot\|_{k,G}$ and $|\cdot|_{k,G}$ be the norm and semi-norm of $H^k(G;X)$, respectively. The space $H(\Div,G;\Snn)$ is equipped with the standard $H(\Div)$ norm
$
\|\tau\|_{H(\Div, G)}^2:=\|\tau\|_{0,G}^2+\|\Div\tau\|_{0,G}^2.
$
Denote the standard $L^2$ inner products on $G$ by~$(\cdot,\cdot)_G$, where the subscript $G$ is omitted if $G = \Omega$, and the duality pairing between $H^{-1/2}(C;\Rn)$ and $H^{1/2}(C;\Rn)$ by $\langle{\cdot,\cdot}\rangle_{C}$ for any given curve $C$.
Let  $\bP_{k}(G; X)$ be the polynomial space taking values in $X$ with degrees not larger than $ k$.

Let $\mathcal{T}_h$ be a regular triangulation of $\Omega$ with $h:=\max_{K\in \mathcal{T}_h}h_K$ and $h_K:= \text{diam}(K)$. By an artificial division $\mathcal{T}_h=\To\cup \Tt$ with $\To\cap \Tt=\emptyset$, denote the unions of all elements in $\To$ and $\Tt$ by $\Omo$ and $\Omt$, respectively.
Let $\Ga$ denote the interface between $\Omo$ and $\Omt$, and $\Gai=\partial\Omi\setminus\Ga$. 
Denote the sets of all interface elements of $\To$ and $\Tt$ sharing at least one $(n-1)$-dimensional face with $\Ga$ by  $\ToGa$ and $\TtGa$, respectively, see Fig~\ref{Fig:Diagram}. 
%In this paper, without loss of generality, we assume that the ratio between the diameter and the minimum width of $\Omega_h^+$ is $O(1)$ and that $\Gamma_h^+\neq \emptyset$. These assumptions are made to remove the dependence on the mesh size $h$ when applying Korn’s inequality in the subsequent analysis \cite{Horgan1995}.
% we employ Hu–Zhang finite elements only in a small region where stress concentration occurs.
To enhance computational efficiency, $\Omt$ may be chosen as a subdomain of size~$O(h)$, consisting of only a few layers of elements containing stress concentration. 
This enables the following assumption on $\Omo$, which is easily satisfied.
\begin{figure}[htbp]
	\centering
	\begin{tikzpicture}[scale=2,
	dashedline/.style={dashed, thick},
	]
	% 定义正六边形的六个顶点
	\coordinate (A) at (0:1.5);
	\coordinate (B) at (60:1.5);
	\coordinate (C) at (120:1.5);
	\coordinate (D) at (180:1.5);
	\coordinate (E) at (240:1.5);
	\coordinate (F) at (300:1.5);
	\coordinate (O) at (0,0); 
	% 定义线段中点
	\coordinate (A') at (0:0.75);
	\coordinate (B') at (60:0.75);
	\coordinate (C') at (120:0.75);
	\coordinate (D') at (180:0.75);
	\coordinate (E') at (240:0.75);
	\coordinate (F') at (300:0.75);
	\coordinate (AB) at (30:1.299);
	\coordinate (BC) at (90:1.299);
	\coordinate (CD) at (150:1.299);
	\coordinate (DE) at (210:1.299);
	\coordinate (EF) at (270:1.299);
	\coordinate (FA) at (330:1.299);
	%填充颜色
	\fill[gray!80] (BC) --(C') -- (B')--(O) -- (F') -- (E') -- (EF) -- (E)--(O)--(C) --cycle;
	\fill[blue!30] (CD) -- (C')--(D') -- (E') -- (DE) -- (E) -- (O) -- (C) --cycle;
		% 绘制正六边形
	\draw[thick] (A) -- (B) -- (C) -- (D) -- (E) -- (F) -- cycle;
	\draw[thick] (O) -- (A);
	\draw[thick] (O) -- (B);
	\draw[thick] (O) -- (C);
	\draw[thick] (O) -- (D);
	\draw[thick] (O) -- (E);
	\draw[thick] (O) -- (F);
	% 绘制小三角形
	\draw[thick] (A') -- (AB)-- (B') --cycle;
	\draw[thick] (B') -- (BC)-- (C') --cycle;
	\draw[thick] (C') -- (CD)-- (D') --cycle;
	\draw[thick] (D') -- (DE)-- (E') --cycle;
	\draw[thick] (E') -- (EF)-- (F') --cycle;
	\draw[thick] (F') -- (FA)-- (A') --cycle;
	% 标注点
	\coordinate (C1) at (180:1.7);
	\coordinate (C2) at (0:1.7);
	\coordinate (C3) at (120:1.6);
	\coordinate (C4) at (90:0.3);
	\coordinate (C5) at (160:0.4);
	%绘制界面
	\draw[red,line width=1.5pt] (C) -- (O) -- (E);
	
	\node[above] at (C1) {$\mathcal{T}_h^-$};
	\node[above] at (C2) {$\mathcal{T}_h^+$};
	\node[above] at (C3) {$\Gamma_h$};
	\node[above] at (C4) {$\mathcal{T}_{\Gamma,h}^+$};
	\node[above] at (C5) {$\mathcal{T}_{\Gamma,h}^-$};
	\end{tikzpicture}
	\caption{The diagram of local mesh}
	\label{Fig:Diagram}
\end{figure}
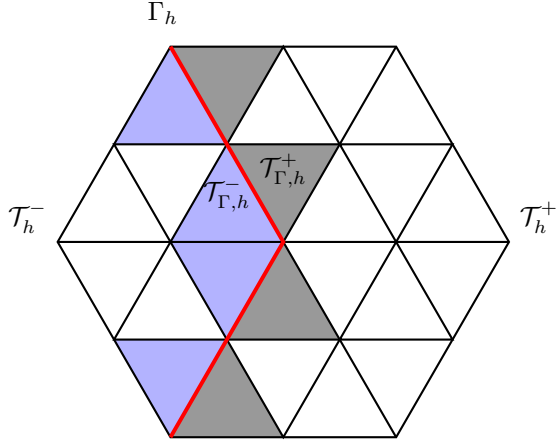	
\begin{assumption}\label{Ass:kroncon}
The domain $\Omega_h^+$ can be decomposed into finitely many connected components. Each connected component $\omega$ is a polygonal domain with $\omega\cap \Gamma_h^+ \neq \emptyset$ and $\text{diam}(\omega)/\text{diam}(B_{\omega}) = O(1)$, where $B_\omega$ is the largest ball contained in $\omega$.
% In particular, to enhance computational efficiency, $\Omt$ may be chosen as a subdomain of size $O(h)$, consisting of only a few layers of mesh elements, provided that the following inequality holds on the remaining part $\Omo$:
  % \begin{equation}
  %   \|\vo\|_{1,\Omo}\leq C\|\varepsilon(\vo)\|_{0,\Omo},
  % \end{equation}
  % where $C$ is a constant independent of $h$. 
\end{assumption}
%Thanks to the assumption that the ratio between the diameter and the minimum width of $\Omega_h^+$ is $O(1)$ and that $\Gamma_h^+\neq \emptyset$ in Section~\ref{notations}, the constant in Korn's inequality \cite{bernner2008the,Horgan1995} on $\Omo$ is independent of $h$. That is 
%\begin{equation}\label{kroncon}
%\|\vo\|_{1,\Omo}\lesssim\|\varepsilon(\vo)\|_{0,\Omo}.
%\end{equation}

% Divide the domain $\Omega$ into two non-overlapping subdomains $\Omo$ and $\Omt$, where $\Omt$ covers the region with stress concentration. Let $\Gamma$ denote the interface between $\Omo$ and $\Omt$, and $\Gai=\partial\Omi\setminus\Gamma$. Let $v^+ = v|_\Omo$ and $v^- = v|_\Omt$ be the restriction of function $v$ to $\Omo$ or $\Omt$, respectively.
% Let $\mathcal{T}_h$ be a regular triangulation of $\Omega$ with no elements cut by the interface, which leads to subtriangulations $\To$ and $\Tt$ of the subdomains $\Omo$ and $\Omt$, respectively.
Let $\mathcal{F}_h$ (resp.~$\mathcal{F}_h^{\pm}$) be the set of all~$(n-1)$-dimensional faces of $\mathcal{T}_h$ (resp.~$\mathcal{T}_h^\pm$) and let $\mathcal{F}_{h,I}$(resp.~$\mathcal{F}_{h,I}^{\pm}$) be the corresponding set of all $(n-1)$-dimensional interior faces.
The set of all~$(n-1)$-dimensional faces contained on $\Ga$ is defined by $\FhGa$.
Denote the diameter of $F\in\mathcal{F}_h$ by $h_F$, and the unit normal of $F$ pointing from the element with larger index to the one with smaller index by~$n_F$. Let 
$v^+ = v|_\Omo$ and~$v^- = v|_\Omt$. Define the jump
$$
[\vt]:=\begin{cases}
    \vt|_{K_1}- \vt|_{K_2},& \mbox{ if } F\in\mathcal{F}^-_{h,I},\\
    \vt,& \mbox{ if } F\in\Ft\setminus\mathcal{F}^-_{h,I}.
\end{cases}
$$
where $K_1$ and $K_2$ are the elements sharing $F$ with the index of $K_1$ larger than that of $K_2$. 
Throughout the paper, let $C$ be a generic positive constant independent of the mesh size $h$, and refer to different values at different places. For ease of presentation, we shall use the symbol $A\lesssim B$ to denote that $A\leq CB$.

%\section{Conforming-Mixed Finite Element Method for linear elasticity problem}
% \section{Hu--Zhang and Lagrange Hybrid Finite Element Method for linear elasticity problem}
\section{A coupling mixed and Lagrange finite elements method}
\label{con-mix}
In this section, we propose a coupling finite element method for the linear elasticity equations~\eqref{intro::linearEquation}, and analyze the well-posedness and optimal priori error estimate.

\subsection{Weak formulation}
To begin with the discrete coupling finite element method, we introduce the coupling weak formulation, which adopts the primal form of elasticity problem on $\Omo$ and the mixed form on $\Omt$. The following function spaces are defined on $\Omo$ and $\Omt$, respectively:
$$
V^+=\{v\in H^1(\Omo;\Rn):~v|_{\Gao}=0\},\
\Sigt=H(\Div,\Omt;\Snn), \  \Vt=L^2(\Omt;\Rn). 
$$  
The primal weak formulation on $\Omo$ is to find $\uo\in\Vo$ such that for all $ \vo\in\Vo$,
\begin{equation}
\label{Omo:weakform}
(\Al^{-1}\varepsilon(\uo),\varepsilon(\vo))_{\Omo}-\langle\sigo \no,\vo\rangle_{\Ga}=(f,\vo)_{\Omo},
\end{equation}
where $\no$ is the unit normal vector of $\Ga$ pointing from $\Omo$ to $\Omt$.
The mixed weak formulation on $\Omt$ is to find $(\sigt,\ut)\in\Sigt\times\Vt$ such that for all $(\taut,\vt)\in\Sigt\times\Vt$,
\begin{equation}
\label{Omt:weakform}
\begin{cases}
(\Al\sigt,\taut)_{\Omt}+(\Div\taut,\ut)_{\Omt}-\langle\taut \nt,\ut\rangle_{\Ga}=0,\\
(\Div\sigt,\vt)_{\Omt}=-(f,\vt)_{\Omt}, 
\end{cases}  
\end{equation}
where $\nt=-\no$ is the unit normal vector of $\Ga$ pointing from $\Omt$ to $\Omo$.

% Assume that $u\in H^1(\Omega;\Rn)$ and $\sigma\in H(\Div,\Omega;\Snn)$ are the solutions of equation \eqref{intro::linearEquation}, the following conditions hold on the interface $\Ga$,
Note that  the solution $(u,\sigma)\in H^1(\Omega;\Rn)\times H(\Div,\Omega;\Snn)$ of equation~\eqref{intro::linearEquation} is also the solution of the corresponding interface problem with the following interface conditions on $\Ga$
\begin{equation}
\label{Gamma:continus}
\uo=\ut ~\mbox{ and }~ \sigo \no + \sigt \nt=0. 
\end{equation}
% Note that the normal component $\sigt \nt$ acts as the natural boundary condition for the weak formulation \eqref{Omo:weakform} on $\Omo$, and $\uo$ acts as the natural boundary condition for the weak formulation \eqref{Omt:weakform} on $\Omt$. This leads to a new continuous weak formulation on $\Omega$:This generates a global weak formulation on $\Omega$:
Similar to the idea in \cite{wieners1998coupling,hu2024direct},  the continuous formulation seeks~$(\uo,\sigt,\ut)\in \Vo\times\Sigt\times \Vt$ such that for all $(\vo,\taut,\vt)\in \Vo\times\Sigt\times \Vt$,
\begin{equation}
\label{mixform:one}
\begin{cases}
(\Al\sigt,\taut)_\Omt+(\Div\taut,\ut)_\Omt-\langle\taut \nt,\uo\rangle_\Ga&=0,\\
(\Div\sigt,\vt)_\Omt&=-(f,\vt)_\Omt,\\
-\langle\sigt \nt,\vo\rangle_{\Ga} -(\Al^{-1}\varepsilon(\uo),\varepsilon(\vo))_{\Omo}&=-(f,\vo)_{\Omo},
\end{cases}
\end{equation}
which is a symmetric perturbed saddle point system and can be rewritten as an equivalent saddle point system
\begin{equation}
\label{mixform:two}
\begin{cases}
a(\uo,\sigt; \vo,\taut) + b(\vo,\taut; \ut) &= (f,\vo)_{\Omo},\\
b(\uo,\sigt; \vt)&=-(f,\vt)_{\Omt},
\end{cases}   
\end{equation}
where the bilinear forms
\begin{equation}
\label{define:ab}
\begin{aligned}
a(\uo,\sigt;\vo,\taut)
&:=(\Al^{-1}\varepsilon(\uo),\varepsilon(\vo))_\Omo+(\Al\sigt,\taut)_\Omt\\
&\quad+\langle\sigt \nt,\vo\rangle_\Ga-\langle\taut\nt,\uo\rangle_\Ga,\\
b(\vo,\taut; \vt)&:=(\Div\taut,\vt)_{\Omt}.
\end{aligned}    
\end{equation}
% {\color{red}
% Thanks to the assumption that the ratio between the diameter and the minimum width of $\Omega_h^+$ is $O(1)$ and that $\Gamma_h^+\neq \emptyset$ in Section~\ref{notations}, the constant in Korn's inequality \cite{bernner2008the,Horgan1995} on $\Omo$ is independent of $h$. That is 
% \begin{equation}\label{kroncon}
% \|\vo\|_{1,\Omo}\lesssim\|\varepsilon(\vo)\|_{0,\Omo}.
% \end{equation}
% }
For any $(\vo,\taut)\in\Vo\times\Sigt$ and $\vt\in\Vt$, define the following norms:
\begin{equation}
\label{define:norms}
\svt{(\vo,\taut)}\svt_1:=\|\varepsilon(\vo)\|_{0,\Omo}+\|\taut\|_{H(\Div,\Omt)}
\quad \mbox{and}\quad
\svt{\vt}\svt_0:=\|\vt\|_{0,\Omt},
\end{equation}
According to~\cite{bernner2008the,Horgan1995}, Assumption \ref{Ass:kroncon} guarantees Korn's inequality on $\Omo$,
\begin{equation}\label{kroncon}
	\|v^+\|_{1,\Omo}\leq C\|\varepsilon(v^+)\|_{0,\Omo},\quad \forall v\in V^+,
\end{equation}
which indicates that $\svt{\cdot}\svt_1$ is a norm on $\Vo\times\Sigt$.

\begin{lemma}
\label{wellponess:conweakform}
The weak formulation \eqref{mixform:two} is well-posed, namely, there exist a unique solution $(\uo,\sigt,\ut)\in\Vo\times\Sigt\times\Vt$ satisfying:
$$
\svt{(\uo,\sigt)}\svt_1+\svt{\vt}\svt_0\lesssim\|f\|_{0,\Omega}.
$$
\end{lemma}
\begin{proof}
By the trace theorem, the Cauchy-Schwarz inequality and Korn's inequality \eqref{kroncon} that
$$
\begin{aligned}
|\langle\sigt \nt,\vo\rangle_\Ga| &\leq
\sum_{F\in\FhGa}|\langle\sigt \nt,\vo\rangle_F|\leq\sum_{F\in\FhGa}\|\sigt\nt\|_{-\frac{1}{2},F}\|\vo\|_{\frac{1}{2},F}\\
&\lesssim(\sum_{K\in\TtGa}\|\sigt\|_{0,K}^2+\|\Div\sigt\|_{0,K}^2)^{\frac{1}{2}}(\sum_{K\in\ToGa}\|\vo\|_{1,K}^2)^{\frac{1}{2}}\\ &\lesssim\|\sigt\|_{H(\Div,\Omt)}\|\varepsilon(\vo)\|_{0,\Omo}.
\end{aligned}
$$
By the above inequality and the definition of norms in \eqref{define:norms}, the bilinear forms~$a(\cdot,\cdot)$ and $b(\cdot,\cdot)$ are continuous. To be specific,
\begin{equation}\label{continousab}
\begin{aligned}
|a(\uo,\sigt;\vo,\taut)|&\lesssim \svt{(\uo,\sigt)}\svt_1\svt{(\vo,\taut)}\svt_1,\\
|b(\vo,\taut;\ut)|&\leq\svt{(\vo,\taut)}\svt_1\svt{\ut}\svt_0.
\end{aligned}    
\end{equation}
Denote the kernel space of $\Vo\times\Sigt$ by
$$
Z:=\{(\vo,\taut)\in\Vo\times\Sigt~:~b(\vo,\taut;\vt)=0,~\mbox{ for all}~\vt\in\Vt\}.
$$
% Note that $\Div\taut=0$ holds for any $(\vo,\taut)\in Z$, which implies that the bilinear form $a(\cdot,\cdot)$ is coercive on $Z$; that is
For any $(\vo,\taut)\in Z$, it holds that $\Div\taut=0$. This implies the coercivity of the bilinear form $a(\cdot,\cdot)$ on $Z$:
\begin{equation}
\label{Kepl}
\begin{aligned}
     a(\vo,\taut;\vo,\taut)
     % &=(\Al^{-1}\varepsilon(\vo),\varepsilon(\vo))_\Omo+(\Al\taut,\taut)_\Omt\\
     % &\geq{\color{red} \frac12\min(\frac{1}{\beta},\alpha)\svt{(\vo,\taut)}\svt_1^2},
     &\geq C\svt{(\vo,\taut)}\svt_1^2.
\end{aligned}
\end{equation}
For each $\vt\in\Vt$. According to the stability result of~\cite{MR3352360,giraul1986finite}, there exists a $\tau\in H^1(\Omega;\Snn)$ such that
$$
\Div\tau|_{\Omo}=0,~\Div\tau|_{\Omt}=\vt,~\text{ and }~\|\tau\|_{1,\Omega}\leq C\|\vt\|_{0,\Omt}.
$$
Taking $\taut=\tau|_{\Omt}$, we obtain $\taut\in\Sigt$ with $\Div\taut=\vt$ and~$\|\taut\|_{H(\Div,\Omt)}\leq C\|\vt\|_{0,\Omt}$.
% {\color{red}{According to \cite{MR1930384,MR3352360}}},
% for any $\vt\in\Vt$, there exists $\taut\in\Sigt$ such that ${\color{red}\Div\taut=\vt}$ and $\|\taut\|_{H(\Div,\Omt)}\leq C\|\vt\|_{0,\Omt}$. 
% By the definition of $b(\ut;\taut,\vt)$, 
Thus the inf-sup condition holds
\begin{equation}\label{infsupcon}
    \inf_{0\neq\vt\in\Vt}\sup_{(\vo,\taut)\in\Vo\times\Sigt}\frac{b(\vo,\taut;\vt)}{\svt{(\vo,\taut)}\svt_1\svt{\vt}\svt_0}\geq \inf_{0\neq\vt\in\Vt}\frac{\svt{\vt}\svt_0^2}{C\svt{\vt}\svt_0^2}=\frac{1}{C}.
\end{equation}
Together with the boundedness \eqref{continousab}, the coercivity \eqref{Kepl}, and the inf-sup condition~\eqref{infsupcon}, it follows from \cite[Theorem 4.2.3]{MR3097958} that
$$
\begin{aligned}
% \svt{(\sigt,\uo)}\svt_1&\leq CC_{\alpha,\beta}\|f\|_{0,\Omega},\\
% \svt{u^-}\svt_0&\leq C C_{\alpha,\beta}\|f\|_{0,\Omega},
\svt{(\uo,\sigt)}\svt_1+\svt{\vt}\svt_0\lesssim\|f\|_{0,\Omega},
\end{aligned}
$$
which completes the proof.
\end{proof}

% \subsection{The Hu--Zhang and Lagrange hybrid finite element method}
\subsection{The coupling finite element method} 
Based on the weak formulation~\eqref{mixform:two}, a coupling finite element method applying the $P_k$ Hu--Zhang element on $\Omt$ and the $P_{k+1}$ Lagrange element on $\Omo$ is proposed, which maintains both solution accuracy and implementation efficiency.

%Throughout subsequent sections, unless otherwise specified, we consistently adopt the convention that $k$ is a positive integer satisfying $k\geq n+1$, where~$n=2,3$ is the space dimension.

\subsubsection{Discrete formulation} 

The Hu--Zhang element with degree $k$ ($k\geq n+1$) in~\cite{MR3352360, MR3301063, hu2015familyconformingmixedfinite} is employed here to approximate~$\sigt$ in
\begin{equation}
    \label{defSigth}
   \begin{aligned}
\Sigth&:=\{\sigma\in H(\Div,\Omt;\Snn)\,:~  \sigma=\sigma_c+\sigma_b,  \sigma_c\in H^1(\Omt;\Snn),  \\
&\qquad \sigma_c|_K\in \bP_k(K; \Snn), \,  \sigma_b|_K\in\Sigma_{K, k, b},~\mbox{ for all}~K\in\Tt\},
\end{aligned} 
\end{equation}
with bubble function space:
$$
\Sigma_{K, k, b} =\sum_{0\leq i<j\leq n}\lambda_i\lambda_j\bP_{k-2}(K;\bR)t_{i,j}t_{i,j}^T,
$$
where $\lambda_i$ is the associated barycentric coordinate corresponding to vertex $\bx_i$, and  $t_{i,j}$ is the associated unit tangent vectors of the edge of $K$ with vertices $\bx_i$ and $\bx_j$.
The discrete displacement spaces $\Vth$ and $\Voh$ for variables $\ut$ and~$\uo$, respectively, are chosen to be
\begin{align}
 \label{defVth}
 \Vth:&=\{v\in \Vt~:~v|_K\in \bP_{k-1}(K;\Rn),~\mbox{ for all}~K\in\Tt\},\\
 \label{defVoh}
 \Voh:&=\{v\in \Vo~:~v|_K\in\bP_{k+1}(K;\Rn),~\mbox{ for all}~K\in\To\}.
\end{align}
Then, our coupling finite element method for problem \eqref{mixform:two}, employing the Hu--Zhang element and the Lagrange element, seeks $(\uoh,\sigth,\uth)\in \Voh\times\Sigth\times\Vth$, such that for all $(\voh,\tauth,\vth)\in\Voh\times\Sigth\times\Vth$,  
\begin{equation}
\label{dismixform}
\begin{cases}
a(\uoh,\sigth; \voh,\tauth) + b(\voh,\tauth; \uth) &= (f,\voh)_{\Omo},\\
b(\uoh,\sigth; \vth)&=-(f,\vth)_{\Omt}.
\end{cases}   
\end{equation}

\subsubsection{Well-posedness and error estimates} 
The following lemma proves the well-posedness of the proposed coupling finite element method in \eqref{dismixform}.
\begin{lemma}
\label{diswell-posedness}
The discrete problem \eqref{dismixform} admits a unique solution $(\uoh,\sigth,\uth)\in\Voh\times\Sigth\times\Vth$ satisfying
$
\svt{(\uoh,\sigth)}\svt_1+\svt{\uth}\svt_0\lesssim\|f\|_{0,\Omega}.
$
% where the constant $C_{\alpha,\beta}$ depends on $\alpha,\beta$, but is independent of $h$.
\end{lemma}
\begin{proof}
The continuity of the bilinear forms $a(\cdot,\cdot)$ and $b(\cdot,\cdot)$ follows from the fact that $\Voh\times\Sigth\times\Vth$ is a conforming subspace of $\Vo\times\Sigt\times\Vt$ and~\eqref{continousab}.
% {\color{blue}The proof for the continuity of the bilinear $a(\cdot,\cdot)$ is trivial.}
By the definition, the discrete space $\Vth$ consists of piecewise polynomials with degrees at most $k-1$, while $\Sigth$ consists of piecewise polynomials with degrees at most $k$,  it follows that
$\Div\Sigth\subset\Vth.$
Thus, the discrete kernel space
$$
\begin{aligned}
\Zh&=\{(\voh,\tauth)\in\Voh\times\Sigth~:~b(\voh,\tauth;\vth)=0,~\mbox{ for all }\,\vth\in\Vth\}\\
&=\{(\voh,\tauth)\in\Voh\times\Sigth~:~\Div\tauth|_{K}=0,~\mbox{ for all }\,K\in\Tt\}
\end{aligned}
$$
is a subset of $Z$. This, together with \eqref{Kepl}, implies that the bilinear form $a(\cdot,\cdot)$ is coercive on $\Zh$. 
% namely,
% \begin{equation}
% \label{disKepl}
% \begin{aligned}
% a(\voh,\tauth;\voh,\tauth)
% &\geq\frac12\min(\frac{1}{\beta},\alpha)\svt{(\voh,\tauth)}\svt_1^2,
% \end{aligned}
% \end{equation}
% for all $(\voh,\tauth)\in \Zh$.

According to \cite[Theorem 3.1]{MR3352360} and the same argument to the proof of Lemma \ref{wellponess:conweakform}, for any $\vth\in\Vth$, there exists $\tauth\in\Sigth$ such that $\Div\tauth=\vth$ and $\|\tauth\|_{H(\Div,\Omt)}\leq C\|\vth\|_{0,\Omt}$. Thus, the discrete inf-sup condition holds:
\begin{equation}
\inf_{0\neq\vth\in\Vth}\sup_{(\voh,\tauth)\in\Voh\times\Sigth}\frac{b(\voh,\tauth;\vth)}{\svt{(\voh,\tauth)}\svt_1\svt{\vth}\svt_0}\geq %\inf_{0\neq\vth\in\Vth}\frac{\svt{\vth}\svt_0^2}%{C\svt{\vth}\svt_0^2}=
\frac{1}{C}.
\end{equation}
This implies that
$$
\svt{(\uoh,\sigth)}\svt_1+\svt{\vth}\svt_0\lesssim\|f\|_{0,\Omega},
$$
which completes the proof.
\end{proof}

A combination of the coercivity condition, the inf-sup condition, the standard mixed finite element theory \cite[Theorem 5.2.2]{MR3097958}, and the canonical interpolation error estimates of $\Sigth$, $\Vth$, and $\Voh$ \cite{MR3352360, bernner2008the, MR1011446}, yields the following error estimate.
\begin{theorem}\label{Th:main1}
\label{diserror}
Let $(\uo,\sigt,\ut)\in\Vo\times\Sigt\times\Vt$ be the solution of problem~\eqref{mixform:two} and~$(\uoh,\sigth,\uth)\in\Voh\times\Sigth\times\Vth$ be the discrete solution of \eqref{dismixform}. If $\uo\in H^{k+1}(\Omo;\Rn)$, $\sigt\in H^{k+1}(\Omt;\Rn)$ and $\ut\in H^{k}(\Omt;\Rn)$, the following error estimate holds:
% Then, for $k\geq n+1$, we have
\begin{equation}
\label{error:one}
\begin{aligned}
\|\varepsilon(\uo)-\varepsilon(\uoh)\|_{0,\Omo}&+\|\sigt-\sigth\|_{H(\Div,\Omt)}+\|\ut-\uth\|_{0,\Omt}\\
&\lesssim h^{k}(\|\uo\|_{k+1,\Omo}+\|\sigt\|_{k+1,\Omt}+\|\ut\|_{k,\Omt}).
% \|\uo-\uoh\|_{1,\Omo}&+\|\sigt-\sigth\|_{H(\Div,\Omt)}+\|\ut-\uth\|_{0,\Omt}\\
% &\lesssim h^{k}(\|\uo\|_{k+1,\Omo}+\|\sigt\|_{k+1,\Omt}+\|\ut\|_{k,\Omt}).
\end{aligned}
\end{equation}
\end{theorem}

% \begin{proof}
% The well-posedness results established in Lemma \ref{wellponess:conweakform} and Lemma \ref{diswell-posedness}, together with the standard theory of mixed finite element methods \cite[Theorem 5.2.2]{MR3097958}, immediately yield the following error estimate
% \begin{equation}\label{quaerror}
% \begin{aligned}
% &\svt{(\uo-\uoh,\sigt-\sigth)}\svt_1+\svt{\ut-\uth}\svt_0 \\
% &\quad\lesssim\inf_{(\voh,\tauth,\vth)\in\Voh\times\Sigth\times\Vth}\svt{(\uo-\voh,\sigt-\tauth)}\svt_1+\svt{\ut-\vth}\svt_0.
% \end{aligned}    
% \end{equation}
% Let $\Ioh$ be the canonical interpolation of Lagrange finite element from $\Vo$ to $\Voh$, and $\Pth$ be the local $L^2$ projection operator form $\Vt$ to $\Vth$. The following estimates hold:
% \begin{align}
% &\|v-\Ioh v\|_{1,\Omo}\lesssim h^{l}\|v\|_{l+1,\Omo}~\mbox{for all}~v\in H^{l+1}(\Omo;\Rn)~\mbox{and}~0\leq l\leq k+1, \label{Ioherror}   \\
% &\|v-\Pth v\|_{0,\Omt}\lesssim h^{k}\|v\|_{k,\Omt} ~\mbox{for all}~v\in H^k(\Omt;\Rn). \label{Ptherror}
% \end{align}
% Denote $\Ith$ by the Scott-Zhang \cite{MR1011446} interpolation operator from $\Sigt$ to $\Sigth$, satisfying
% \begin{equation}\label{Itherror}
% \|\sigt-\Ith\sigt\|_{0,\Omt}+h|\sigt-\Ith\sigt|_{H(\Div,\Omt)}\lesssim h^{k+1}\|\sigt\|_{k+1,\Omt}.
% \end{equation}
% Taking $\voh=\Ioh\uo$, $\tauth=\Ith\taut$ and $\vth=\Ith\ut$ in \eqref{quaerror}, we can obtain \eqref{error:one}.
% \end{proof}
\begin{remark}
% Although the choice of the subdomain $\Omega_h^{\pm}$ depends on the mesh, the discrete displacement and stress still achieve global convergence of order $k$, and this convergence rate is independent of the size of $\Omega_h^{\pm}$, provided that the exact solution satisfies $u\in H^{k+2}(\Omega)$.
The theorem above implies that the $k$-th order global convergence rate is guaranteed under a quite flexible choice of $\Omo$ and $\Omt$ if the exact solution~$u\in H^{k+2}(\Omega;\Rn)$. This offers the convenience to choose suitable domain divisions to meet computational requirements.
\end{remark}

\subsubsection{Optimal error estimate of $\sigth$ and $\uoh$}
In this subsection, an improved and optimal error estimate of $\uoh$ in both energy norm and $L^2$ norm and $\sigth$ in~$L^2$ norm is proved. Furthermore, two superconvergent results of $\uth$ are provided.

To begin with, define the mesh-dependent norms
$$
\begin{aligned}
\|\tauth\|_{0,h,\Omt}^2&:=\|\tauth\|_{0,\Omt}^2+\sum_{F\in\Ft}h_F\|\tauth n_F\|_{0,F}^2,\\
|\vth|_{1,h,\Omt}^2&:=\|\varepsilon_h(\vth)\|_{0,\Omt}^2+\sum_{F\in\mathcal{F}_h^-}h^{-1}_F\|[\vth]\|^2_{0,F},\\
\svt{(\voh,\tauth)}\svt&:=\|\varepsilon(\voh)\|_{0,\Omo}+\|\tauth\|_{0,h,\Omt}.
% \svt{(\voh,\tauth)}\svt&:=\|\voh\|_{1,\Omo}+\|\tauth\|_{0,h,\Omt}.
\end{aligned}
$$
By \cite[Lemma 3.3.]{MR3787386},
$$
|\vth|_{1,h,\Omt}\lesssim\sup_{\tauth\in\Sigth}\frac{(\Div\tauth,\vth)_{\Omt}}{\|\tauth\|_{0,h,\Omt}},~\mbox{ for all} ~\vth\in\Vth,
$$
which implies the following discrete inf-sup condition 
\begin{equation}\label{weak::BB}
|\vth|_{1,h,\Omt}\lesssim\sup_{(\voh,\tauth)\in\Voh\times\Sigth}\frac{b(\voh,\tauth;\vth)}{\svt{(\voh,\tauth)}\svt},~\mbox{ for all}~ \vth\in\Vth.
\end{equation}
Moreover, the following coercivity condition holds on the whole space: 
\begin{equation}\label{weak::cor}
\svt{(\vo_h,\taut_h)}\svt^2\lesssim a(\vo_h,\taut_h;\vo_h,\taut_h),~\mbox{ for all }\, (\voh,\tauth)\in\Voh\times\Sigth. 
\end{equation}
% To obtain the optimal error estimate for $\uo$ in the $H^1$ norm and $\sigt$ in the $L^2$ norm, we define
% $$
% \svt{(\vo,\taut)}\svt:=\|\vo\|_{1,\Omo}+\|\taut\|_{0,\Omt},~\mbox{for all}~ (\vo,\taut)\in\Vo\times\Sigt,
% $$
% and it satisfies the coercivity condition
% \begin{equation}\label{weak::cor}
% \svt{(\vo,\taut)}\svt^2\lesssim a(\vo,\taut;\vo,\taut).    
% \end{equation}
% For any $\vth\in\Vth$, define
% $$
% |\vth|_{1,h,\Omt}=\|\varepsilon_h(\vth)\|_{0,\Omt}^2+\sum_{F\in\mathcal{F}_h^-}h^{-1}_F\|[\vth]\|^2_{0,F}
% $$
% By \cite[(3.1)]{MR3787386}, for all $\vth\in\Vth$, $$|\vth|_{1,h,\Omt}\lesssim\sup_{\tauth\in\Sigth}\frac{(\Div\tauth,\vth)_{\Omt}}{\|\tauth\|_{0,\Omt}}$$
% thus the following discrete inf-sup condition holds:
% \begin{equation}\label{weak::BB}
% |\vth|_{1,h,\Omt}\lesssim\sup_{(\voh,\tauth)\in\Voh\times\Sigth}\frac{b(\voh,\tauth;\vth)}{\svt{(\voh,\tauth)}\svt}, \qquad \text{for all}\, \vth\in\Vth.
% \end{equation}
The stability follows from the coercivity condition \eqref{weak::cor} and the discrete inf-sup condition \eqref{weak::BB}.

\begin{lemma}\label{weak:stablity}
It follows that for any $(\uoh,\sigth,\uth)\in\Voh\times\Sigth\times\Vth$,
\begin{equation}\label{weak:stablityeq} 
\svt{(\uoh,\sigth)}\svt+|\uth|_{1,h,\Omt}\lesssim\sup_{(v_h^+,\tau_h^-,v_h^-)\in\Voh\times\Sigth\times\Vth}\frac{\mathbb{A}(\uoh,\sigth,\uth;v_h^+,\tau_h^-,v_h^-)}{\svt{(v_h^+,\tau_h^-)}\svt+|v_h^-|_{1,h,\Omt}}
\end{equation}
with 
$\mathbb{A}(\uoh,\sigth,\uth;v_h^+,\tau_h^-,v_h^-):=a(\uoh,\sigth;v_h^+,\tau_h^-) + b(v_h^+,\tau_h^-; \uth)+b(\uoh,\sigth; v_h^-).$
\end{lemma}

Let $\Ioh:H^1(\Omo;\Rn)\to\Voh$ be the Scott-Zhang interpolation operator \cite{MR1011446}. For any $\uo\in H^{k+2}(\Omo;\Rn)$, it follows from \cite[Theorem 4.1]{MR1011446} that $\Ioh$ satisfies the following error estimate,
\begin{equation}\label{Ioherr}
\Big(\sum_{K\in\To}h_{K}^{2(m-k-2)}\|\Ioh\uo-\uo\|_{m,K}^2\Big)^{1/2}\lesssim \|\uo\|_{k+2,\Omo},~\mbox{ for }\,m=0,1.
% \|\Ioh\uo-\uo\|_{0,\Omo}+h\|\Ioh\uo-\uo\|_{1,\Omo}\lesssim h^{k+2}\|\uo\|_{k+2,\Omo}.
\end{equation}
Let $\Pth:L^2(\Omt;\Rn)\to\Vth$ denote the $L^2$ orthogonal projection, and let $\Pits:H^1(\Omt;\Snn)\to\Sigth$ be the interpolation operator defined in~\cite[Remark 3.1]{MR3352360}. These operators satisfy the following commuting property:
\begin{equation}\label{weak:PitsPro}
    \Div(\Pits\sigt)=\Pth(\Div\sigt).
\end{equation}
Moreover, if $\sigt\in H^{k+1}(\Omt;\mathbb{S})$, then $\Pits$ admits the interpolation error estimate
\begin{equation}\label{pitserr}
\|\Pits\sigt-\sigt\|_{0,\Omo}+h\|\Pits\sigt-\sigt\|_{1,\Omo}\lesssim h^{k+1}\|\sigt\|_{k+1,\Omt}.
\end{equation}

The following lemma characterizes the interpolation error of the exact solution with respect to the bilinear form $a(\cdot,\cdot)$.
\begin{lemma}
For $\uo\in H^{k+2}(\Omo;\mathbb{R}^n)$ and $\sigt\in H^{k+1}(\Omt;\mathbb{S})$, it holds that
\begin{equation}\label{weak::aestimate}
\begin{aligned}
     \sup_{(\voh,\tauth)\in\Voh\times\Sigth}&\frac{a(\Ioh\uo-\uo,\Pits\sigt-\sigt;\voh,\tauth)}{\svt(\voh,\tauth)\svt}\\
     &\qquad\qquad\lesssim h^{k+1}(\|\uo\|_{k+2,\Omo}+\|\sigt\|_{k+1,\Omt}).
\end{aligned}
\end{equation}
\end{lemma}
\begin{proof}
Since $\uo\in H^{k+2}(\Omo;\mathbb{R}^n)$ and $\sigt\in H^{k+1}(\Omt;\mathbb{S})$, $\langle\cdot,\cdot\rangle_{\Ga}$ can be regarded as the standard $L^2$ inner product $(\cdot,\cdot)_{\Ga}$. By the definition of $a(\cdot,\cdot)$,
\begin{equation}\label{weak::aestimate1}
 \begin{aligned}
|a(&\Ioh\uo-\uo,\Pits\sigt-\sigt;\voh,\tauth)|\\
 &\lesssim\|\varepsilon(\Ioh\uo)-\varepsilon(\uo)\|_{0,\Omo}\|\varepsilon(\voh)\|_{0,\Omo}+\|\Pits\sigt-\sigt\|_{0,\Omt}\|\tauth\|_{0,\Omt}\\
% &\lesssim\|\Ioh\uo-\uo\|_{1,\Omo}\|\voh\|_{1,\Omo}+\|\Pits\sigt-\sigt\|_{0,\Omt}\|\tauth\|_{0,\Omt}\\
&+|\big((\Pits\sigt-\sigt)\nt,\voh\big)_\Ga|+|\big(\tauth\nt,\Ioh\uo-\uo\big)_\Ga|.
\end{aligned}
\end{equation}
For any $F\in\FhGa$, there exists an element $\Ko\in\ToGa$, such that $F$ is a face of~$\Ko$. Define the local rigid motion space $R(\Ko):=\{v\in H^1(\Ko;\Rn)~:~\varepsilon(v)=0\}$. For any $v\in H^1(\Ko;\Rn)$, let~$\pi$ be the interpolation operator from $H^1(\Ko;\Rn)$ onto~$R(\Ko)$ defined by 
\begin{align*}
	 \int_{K^+}(v-\pi v)\text{d}x =0~\text{ and }~\int_{K^+}\nabla\times(v-\pi v)\text{d}x =0,
\end{align*}
here, $\nabla\times$ denotes the standard curl operator. The interpolation operator $\pi$ satisfies the following estimate (c.f.~(3.3)-(3.4) in~\cite{brener2004korns}):
\begin{equation*}
    h_{K^+}^{-1}\|v-\pi v\|_{0,K^+}\lesssim |v-\pi v|_{1,K^+}\lesssim \|\varepsilon(v-\pi v)\|_{0,K^+}\lesssim \|\varepsilon(v)\|_{0,K^+}.
\end{equation*}
This, combined with the result~$((\Pits\sigt-\sigt)\nt,\pi\voh)_F = 0$ in~\cite[Lemma $3.1$]{MR3352360}, the Cauchy-Schwarz inequality, and the trace theorem, yields %\cite[Theorem 1.6.6]{Barna2010The}
$$
\begin{aligned}
\big\vert\big((\Pits\sigt-\sigt)\nt,\voh\big)_F\big\vert&=\big\vert\big((\Pits\sigt-\sigt)\nt,\voh-\pi\voh\big)_F\big\vert\\
&\leq h_F^{\frac{1}{2}}\|(\Pits\sigt-\sigt)\nt\|_{0,F}\cdot h_F^{-\frac{1}{2}}\|\voh-\pi\voh\|_{0,F}\\
&\lesssim h_F^{\frac{1}{2}}\|(\Pits\sigt-\sigt)\nt\|_{0,F}\cdot\|\varepsilon(\voh)\|_{0,\Ko}.
% &\lesssim h^{k+1}|\sigt|_{k+1,\Kt}\|\varepsilon(\voh)\|_{0,\Ko}
\end{aligned}
$$
Summarizing the above results across all interfaces, we derive
\begin{equation}\label{ErrorGammaN}
\begin{aligned}
\big\vert\big(&(\Pits\sigt-\sigt)\nt,\voh\big)_\Ga\big\vert\\
&\leq \sum_{F\in\FhGa}\big\vert\big((\Pits\sigt-\sigt)\nt,\voh\big)_F\big\vert\\
% &\lesssim (\sum_{F\in\FhGa}h_F\|(\Pits\sigt-\sigt)\nt\|_{0,F}^2)^{\frac{1}{2}}(\sum_{\Ko\in\ToGa}\|\varepsilon(\voh)\|_{0,\Ko}^2)^{\frac{1}{2}}\\
&\lesssim \big(\|\Pits\sigt-\sigt\|_{0,h,\Omt}+h\|\Pits\sigt-\sigt\|_{1,h,\Omt}\big)\|\varepsilon(\voh)\|_{0,\Omo}. 
\end{aligned}
\end{equation}
% By applying the Cauchy-Schwarz inequality and the trace theorem \cite[Theorem 1.6.6]{Barna2010The}, we obtain
The same argument of the above inequality leads to
\begin{equation}
\label{ErrorIhu}
\begin{aligned}
|(\tauth\nt,&\Ioh\uo-\uo)_\Ga|\leq\sum_{F\in\FhGa}h_F^{\frac{1}{2}}\|\tauth\nt\|_{0,F}\cdot h_F^{-\frac{1}{2}}\|\Ioh\uo-\uo\|_{0,F}\\
&\lesssim \|\tauth\|_{0,h,\Omt}\big(\sum_{K\in\ToGa}h_K^{-2}\|\Ioh\uo-\uo\|_{0,K}^2+\|\Ioh\uo-\uo\|_{1,K}^2\big)^{1/2}.
\end{aligned}
\end{equation}
The last two equations, together with \eqref{weak::aestimate1} and the interpolation error estimates \eqref{Ioherr} and \eqref{pitserr}, lead to the desired estimate \eqref{weak::aestimate}.
\end{proof}
% The following theorem provides the optimal estimation of $\uo$ in $H^1$ norm and $\sigt$ in $L^2$ norm.
\begin{theorem}\label{Th:main2}
Let $(\uo,\sigt,\ut)\in\Vo\times\Sigt\times\Vt$ be the solution of problem~\eqref{mixform:two}, and~$(\uoh,\sigth,\uth)\in\Voh\times\Sigth\times\Vth$ be the discrete solution of \eqref{dismixform}. The following error estimate holds:
\begin{equation}\label{opterror}
 %\|\uo-\uoh\|_{1,\Omt}+\|\sigt-\sigth\|_{0,\Omt}\lesssim h^{k+1}(\|\uo\|_{k+2,\Omo}+\|\sigt\|_{k+1,\Omt}).
\svt{(\uo-\uoh,\sigt-\sigth)}\svt + |\Pth\ut-\uth|_{1,h,\Omt}\lesssim h^{k+1}(\|\uo\|_{k+2,\Omo}+\|\sigt\|_{k+1,\Omt}),
\end{equation}
provided that $\uo\in H^{k+2}(\Omo;\mathbb{R}^n)$ and $\sigt\in H^{k+1}(\Omt;\mathbb{S})$.
% Moreover, when $\Omo$ and $\Omt$ are convex, we have
Moreover, if $\Omega$ is convex, we have
\begin{equation}\label{opterroruo}
   \|\uo-\uoh\|_{0,\Omo} + \|\Pth\ut-\uth\|_{0,\Omt} \lesssim h^{k+2}(\|\uo\|_{k+2,\Omo}+\|\sigt\|_{k+1,\Omt}).
\end{equation}

\end{theorem}
\begin{proof}
A subtraction \eqref{dismixform} from \eqref{mixform:two} yields
% yields the error equation
\begin{equation} \label{weak:errorequation}
\begin{cases}
     a(\uo-\uoh,\sigt-\sigth;\voh,\tauth)+b(\voh,\tauth;\ut-\uth)=0,\\
    b(\uo-\uoh,\sigt-\sigth;\vth)=0.
\end{cases}    
\end{equation}
%for all $(\voh,\tauth,\vth)\in\Voh\times\Sigth\times\Vth$.
The second equation of \eqref{weak:errorequation} and the commuting property \eqref{weak:PitsPro} imply
\begin{equation}\label{ab1}
   b(\Ioh\uo-\uoh,\Pits\sigt-\sigth;\vth)=b(\uo-\uoh,\sigt-\sigth;\vth)=0. 
\end{equation}
The first equation of \eqref{weak:errorequation} and the $L^2$ orthogonality of $\Pth$ give
$$
\begin{aligned}
b(\voh,\tauth;\Pth\ut-\uth)&=b(\voh,\tauth;\ut-\uth)
=-a(\uo-\uoh,\sigt-\sigth;\voh,\tauth).
\end{aligned}
$$
Consequently, 
\begin{equation}\label{ab2}
\begin{aligned}
     &a(\Ioh\uo-\uoh,\Pits\sigt-\sigth;\voh,\tauth)+b(\voh,\tauth;\Pth\ut-\uth)\\
    =&a(\Ioh\uo-\uo,\Pits\sigt-\sigt;\voh,\tauth).
\end{aligned}
\end{equation}
% Combining the last two equations, it holds that
A combination of the definition of $\mathbb{A}$ with \eqref{ab1} and \eqref{ab2}, yields
$$
\begin{aligned}
    \mathbb{A}(\Ioh\uo&-\uoh,\,\Pits\sigt-\sigth,\,\Pth\ut-\uth;\,\voh,\tauth,\vth)\\
    &=a(\Ioh\uo-\uoh,\Pits\sigt-\sigth;\voh,\tauth)+b(\voh,\tauth;\Pth\ut-\uth)\\
    &\quad +\,b(\Ioh\uo-\uoh,\Pits\sigt-\sigth;\vth)\\
    &=a(\Ioh\uo-\uo,\Pits\sigt-\sigt;\voh,\tauth),
\end{aligned}
$$
which, together with \eqref{weak:stablityeq} and \eqref{weak::aestimate}, implies
$$
\svt{(\Ioh\uo-\uoh,\Pits\sigt-\sigth)}\svt+|\Pth\ut-\uth|_{1,h,\Omt}\lesssim
h^{k+1}(\|\uo\|_{k+2,\Omo}+\|\sigt\|_{k+1,\Omt}).
$$
The estimate \eqref{opterror} follows from the result above together with the interpolation error estimates \eqref{Ioherr} and \eqref{pitserr}. The error estimate \eqref{opterroruo} can be derived by using the duality argument as in \cite{MR771029,MR1086845}.
\end{proof}

\begin{remark}
	The use of Lagrange element of order $k+1$ guarantees optimal convergence rates for all the three variables, whereas employing Lagrange elements of other orders may degrade the accuracy of certain variables, while still preserving the overall consistency of the coupling method.  
\end{remark}

\begin{remark}\label{postremark}
Based on the superconvergent results of $\uth$ in \eqref{opterror} and \eqref{opterroruo}, we can construct a superconvergent displacement from $(\sigth,\uth)$, using the post-processing algorithm  in \cite[Section 3.5]{MR3787386}.
%it is possible to improve the $L^2$ error of $u^-$ to the order $k + 2$. %there by obtaining an overall $L^2$ error estimate of $u$ at order $k + 2$. 
To be specific, let
$$
\Vthx=\{v\in L^2(\Omt,\Rn)~:~ v|_K\in \bP_{k+1}(K,\Rn),~\mbox{ for all }~K\in\Tt\}.
$$
The post-processing algorithm finds  $\uthx\in\Vthx$ and $\phi_h\in\Vth$, such that for all~$K\in\Tt$,
\begin{equation}\label{posteq}
\begin{aligned}
(\varepsilon(\uthx),\varepsilon(v))_K+(v,\phi_h)_K&=(\Al^{-1}\sigth,\varepsilon(v))_K,~~~\mbox{ for all }~v\in\Vthx|_K,\\
(\uthx,\psi)_K&=(\uth,\psi)_K,~~~~~~~~~~~\mbox{ for all }~\psi\in\Vth|_K.
\end{aligned} 
\end{equation}
It follows from the same argument in \cite[Theorem 3.15]{MR3787386} that
$$
|\ut-\uthx|_{1,h,\Omt}\lesssim h^{k+1}(\|\uo\|_{k+2,\Omo}+\|\ut\|_{k+2,\Omt}+\|\sigt\|_{k+1,\Omt}),
$$
provided that $\uo\in H^{k+2}(\Omo;\mathbb{R}^n)$, $\ut\in H^{k+2}(\Omt;\mathbb{R}^n)$ and $\sigt\in H^{k+1}(\Omt;\mathbb{S})$.
Moreover, when $\Omega$ is convex, it holds that
$$
\|\ut-\uthx\|_{0,\Omt} \lesssim h^{k+2}(\|\uo\|_{k+2,\Omo}+\|\ut\|_{k+2,\Omt}+\|\sigt\|_{k+1,\Omt}).
$$
\end{remark}
\section{Numerical Examples}\label{Sec:Num}
% In the following, we present two numerical examples using the hybrid method with Hu--Zhang and Lagrange elements to solve linear elasticity problems. The convergence order is verified for both a square domain and a unit circular domain. Notably, the second example demonstrates that the hybrid method can achieve optimal convergence rates, even when applied to more complex domains and boundary conditions.
This section presents numerical experiments to demonstrate the performance of the proposed coupling finite element method, with a particular focus on its enhanced robustness and improved stress field approximation. 

Unless otherwise specified, we employ uniform mesh refinement to generate the sequence of refined meshes $\{\mathcal{T}_i\}$ from an initial coarse mesh $\mathcal{T}_0$. For a problem with stress concentration, let the concentration set be a point, curve, surface, or subregion. Define the union of all elements in $\mathcal{T}_h$ intersecting this set by $\omega_1^-$, and the set of elements sharing a vertex with any element in $\omega_{i-1}^-$ by $\omega_i^-$.
For brevity, we introduce the notation
$$
\HZ_{k}(\Omt)+\CG_m(\Omo), \quad\text{and, in particular,}\quad \HZ_{k}(\omega_i^-)+\CG_m(\Omega\setminus\omega_i^-),
$$
to denote a coupled method employing the $P_k$ Hu--Zhang element on the subdomain $\Omt$ (with $\omega_i^-$ as a special case) and the $P_m$ Lagrange
element elsewhere.

\subsection{Convergence rates for homogeneous materials}\label{numerexam1}
Consider the linear elasticity problem \eqref{intro::linearEquation} for homogeneous material with parameters $\lambda=1.0$, $\mu=0.5$ on  domain $\Omega=\left(0,1\right)\times\left(0,1\right)$, which is decomposed into two nonoverlapping subdomains
$$
\Omt=\left(0.25,0.75\right)\times\left(0.25,0.75\right)\quad\mbox{and}\quad\Omo=\Omega\setminus\Omt.
$$ 
The initial mesh $\mathcal{T}_0$ consists of $32$ uniform triangular elements, obtained by cutting the unit square using horizontal, vertical and north-east lines.
%Each mesh $\mathcal{T}_i$ is refined into a half-sized mesh uniformly, to get a higher-level mesh $\mathcal{T}_{i+1}$.
The source term $f$ is determined by the exact solution  $u=(\sin\pi x\sin\pi y, \sin\pi x\sin\pi y)^T$.
\begin{figure}[htbp]
	\centering
	\includegraphics[trim=0 0 0 0 , clip, width=0.76\textwidth]{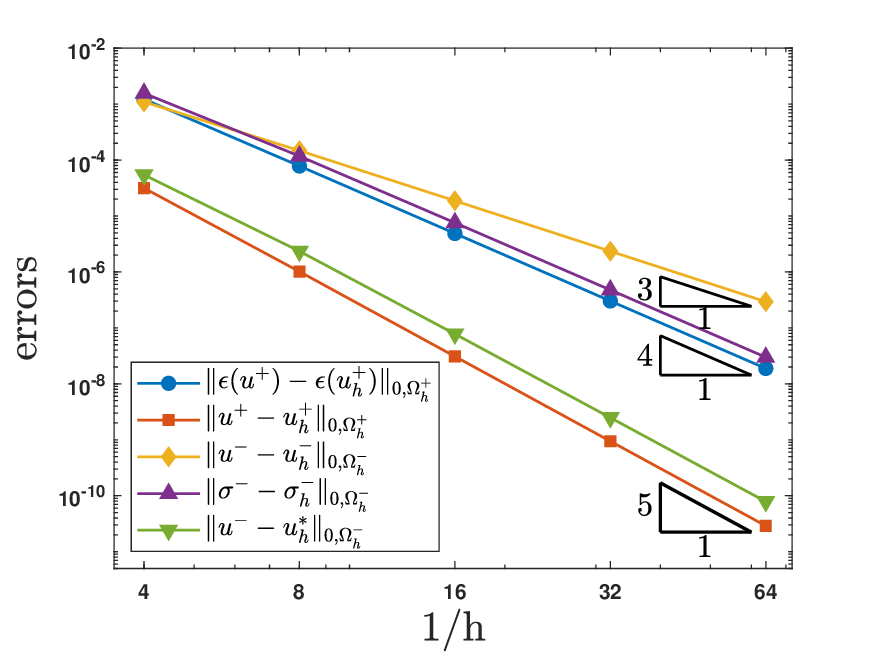}
	\caption{Convergence rates of the coupling method $\HZ_3(\Omega_h^-)+\CG_4(\Omega_h^+)$  for Example \ref{numerexam1}.}
	\label{exam1:fig:res}
\end{figure}

%\begin{figure}[!ht]
%	\centering
%        \subfloat{\includegraphics[trim=5 5 5 5, clip, height=0.2\textheight]{20012002URes.eps}}
%	\subfloat{\includegraphics[trim=5 5 5 5, clip, height=0.2\textheight]{20012002SRes.eps}}
%      \\
% 	\subfloat{\includegraphics[trim=5 5 5 5, clip, height=0.2\textheight]{example1diffLaOrderLagRes.eps}}
%	\subfloat{\includegraphics[trim=5 5 5 5, clip, height=0.2\textheight]{example1diffLaOrderHzuRes.eps}}
%        \\
%        \subfloat{\includegraphics[trim=5 5 5 5, clip, height=0.2\textheight]{example1diffLaOrderHzsRes.eps}}
%	\subfloat{\includegraphics[trim=5 5 5 5, clip, height=0.2\textheight]{example1diffLaOrderHzpostuRes.eps}}
%	\caption{Convergence rates of the coupling $P_3$ Hu--Zhang and $P_k(1\le k\le 5)$ Lagrange finite element method for Example 1.}
%	\label{exam1:fig:res}
%\end{figure}

Figure \ref{exam1:fig:res} presents the numerical results obtained by the $\HZ_3(\Omt) +\CG_4(\Omo)$ method.
The numerical convergence rates for
$\|u^+-u_h^+\|_{0,\Omo}$, $\|\epsilon(u^+)-\epsilon(u_h^+)\|_{0,\Omo}$,
and $\|\sigma^--\sigma_h^-\|_{0,\Omt}$ shown in Figure~\ref{exam1:fig:res} are
$\mathcal{O}(h^5)$, $\mathcal{O}(h^4)$, and $\mathcal{O}(h^4)$, respectively,
in agreement with Theorem~\ref{Th:main2}.
Furthermore, the numerical rate for $\|u^- -u_h^-\|_{0,\Omt}$ is
$\mathcal{O}(h^3)$, confirming the result of Theorem~\ref{Th:main1}.
Moreover, the superconvergence of $\|u^- - u^*_h\|_{\Omt}$ is also observed
in Figure \ref{exam1:fig:res}. Overall, the results demonstrate that the proposed coupling method with finite element spaces $\Sigth \times \Vth \times \Voh$, defined in \eqref{defSigth}--\eqref{defVoh}, achieves optimal convergence rates, consistent with the theoretical analysis.

%The convergence rate does not increase when $k$ is large enough because of the restriction of $P_3$ Hu--Zhang element, which also leads to a deteriorate convergence rate 5.40 for $\|\uo-\uoh\|_{0,\Omo}$ when $k=5$.
%The numerical convergence rates of $\|\ut-\uth\|_{0,\Omt}$ and $\|\sigt-\sigth\|_{0,\Omt}$ are $\mathcal{O}(h^{\min (k+1,3)})$ and $\mathcal{O}(h^{\min (k+0.5,4)})$, respectively. The convergence rate deteriorates for small $k$ since low order Lagrange element is applied, which also causes the deteriorate convergence rate of  $\|u^--u^*_h\|_{0,\Omt}$ when $k\le 3$. 
%The case $k=4$ in Figure \ref{exam1:fig:res} indicates that the pair of $\Sigth\times  \Vth\times  \Voh$  with \eqref{defSigth}-\eqref{defVoh} is optimal, which coincides with the theoretical results. 

%%%%%%%%%%%%%%%%%%%%%%%%%%%%%%%%%%%%%%%%%%%%%%%%%%%%%%%%%%%%%%%%%%%%%%%%%%%%%%%%%%%%%%%%%%%%%%%%%%%%%%%%%%%%%%%%%%%%%%%%%%%%%%%%%%%%%%%%%%%%%%%%%%%%%%%%%%%%%%%%%%%%%%%%%%%%%%%%%%%%%%%%%%%%%%%%%%%%%%%%%%%%%%%%%%%%%%%%%%%%%%%%%%%%%%%%%%%

% \subsection{A solution with a corner singularity}
\subsection{Convergence behavior for low regularity solutions}
Consider linear elasticity problem on L-shaped domain $\Omega=\left(-1,1\right)^2\setminus\big(\left(0,1\right)\times\left(-1,0\right)\big)$ with pure Dirichlet boundary condition.
The exact displacement $u =(u_1,u_2)^{\mathrm T}$ of this problem in polar coordinates~$(r, \theta)$ reads 
% The subdomain adjacent to $(0,0)$ is discretized with Hu--Zhang finite elements, as depicted in Figure \ref{fig:Ldomain}. 
% The exact solution $\bm u =(u_1,u_2)^{\mathrm T}$ is given in polar coordinates $(r, \theta)$ by
$$
\begin{aligned}
	u_1=\frac{1}{2 \mu} r^\gamma&\left(\cos\theta\left(-(1+\gamma)\cos((1+\gamma)\theta)+(\kappa-\gamma)Q\cos((1-\gamma)\theta)\right)\right. \\ - &\left.\sin\theta\left((1+\gamma)\sin((1+\gamma)\theta)-(\kappa+\gamma)Q\sin((1-\gamma)\theta)\right)\right), \\
	u_2=\frac{1}{2 \mu} r^\gamma&\left(\sin\theta\left(-(1+\gamma)\cos((1+\gamma)\theta)+(\kappa-\gamma)Q\cos((1-\gamma)\theta)\right)\right. \\ + &\left.\cos\theta\left((1+\gamma)\sin((1+\gamma)\theta)-(\kappa+\gamma)Q\sin((1-\gamma)\theta)\right)\right), \\
\end{aligned}
$$
with $\kappa=3-4 \nu$, Poisson ratio $\nu=\lambda /(2(\lambda+\mu))$, $Q=-\frac{\cos \left((\gamma+1) \frac{3 \pi}{4}\right)}{\cos \left((\gamma-1) \frac{3 \pi}{4}\right)}$, and $\gamma$ is the solution
of %the equation 
$
\sin \left(\gamma \frac{3 \pi}{2}\right)+\gamma \sin \left(\frac{3 \pi}{2}\right)=0,
$
where approximate values $\gamma=0.5444837367$ and $Q=0.5430755688$ are given in \cite{yi2022locking}.
%According to \cite{yi2022locking}, we choose the numerical values %of $\gamma$ and $Q$ as 
%$\gamma=0.5444837367$ and $Q=0.5430755688$. 
Let $\mu=1$, $\lambda=1$ and the source term~$f$ be determined by the exact solution. 
By~\cite{babuvska1987hp}, the solution $u \in [H^{1+\gamma-\varepsilon}(\Omega)]^2$ and $\sigma \in [H^{\gamma-\varepsilon}(\Omega)]^{2 \times 2}$ for $\varepsilon>0$. Thus, the global convergence rate of stress in the $L^2$ norm is approximately~0.54. Since the exact solution  exhibits singularity at the corner, let the stress concentration set be the point $(0,0)$.%Different orders of the Lagrange element and the Hu--Zhang element and various choices of the region using Hu--Zhang element result in different numerical behavior of coupling method. 
%For ease of presentation,  the notation $P_{k}(i) +P_m$ is adopted for the coupling method, which employs the $P_k$ Hu--Zhang element on the subdomain consisting of $i$ layers of elements surrounding the singularity at~$(0,0)$, and the $P_m$ Lagrange element on the remaining part. 
%The coupling methods $\HZ_{k}(\omega_i^-) +\CG_m(\omega_i^+)$ are adopted for this example.
\subsubsection{Resolution of stress singularity}

The initial mesh $\mathcal{T}_0$ consists of $12$ uniform right triangles, as
depicted in Figure~\ref{fig:Lshapesigma} (a)--(f). As shown in
Figure~\ref{fig:Lshapesigma}, the $\HZ_3(\omega_1^-)+\CG_1(\Omega\setminus\omega_1^-)$ method on $\mathcal{T}_0$ produces results comparable to those on the refined mesh $\mathcal{T}_4$, while the $P_1$ Lagrange method on $\mathcal{T}_0$ fails to capture the singular behavior.
This is further confirmed in Figure~\ref{fig:OriStress}, where the discrete stresses at $(0,0)$ clearly exhibit the singularity resolved by the coupling method. The results demonstrate the superiority of the proposed $\HZ_3(\omega_1^-)+\CG_1(\Omega\setminus\omega_1^-)$ method in resolving stress concentrations.

%The results of Figure~\ref{fig:Lshapesigma} shows that  the~$P_3(1)+P_1$ method resolves the singularity of the stress better than the~$P_1$ Lagrange element on $\mathcal{T}_0$, and exhibits a similar numerical phenomenon to the reference solution on $\mathcal{T}_4$, while  the~$P_1$ Lagrange element on $\mathcal{T}_0$ can not. A more pronounced singular behavior in the stress by the coupling method is observed with the discrete stresses at $(0,0)$ by both methods depicted in Figure~\ref{fig:OriStress}. This highlights the effectiveness of the proposed coupling method in resolving stress concentrations.

%To verify the convergence behavior of the proposed coupling method on the L-shape domain, the Hu–Zhang element is employed on the subdomain formed by the single or several layers of elements adjacent to the singularity at~$(0,0)$, while the remaining part of domain is discretized with the Lagrange element.
% In the L-shaped domain problem (Fig. \ref{fig:Ldomain}), we employ the $P_1$ Lagrange finite element in the blue subdomain and the $P_3$ Hu–Zhang mixed finite element in the red subdomain.
\begin{figure}[htbp]
	\centering
	\subfloat[$\sigma_{xx}$]{\includegraphics[trim=45 20 60 20, clip, width=0.31\linewidth]{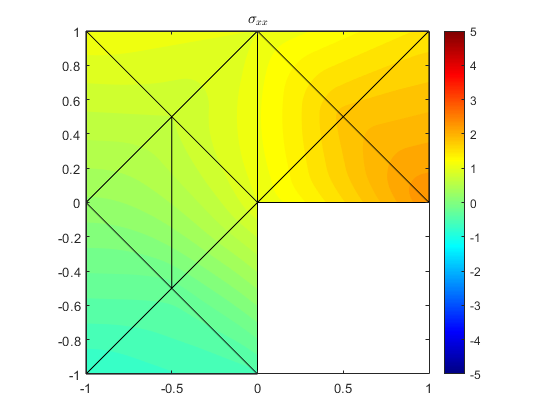}}
	\subfloat[$\sigma_{xy}$]{\includegraphics[trim=45 20 60 20, clip, width=0.31\linewidth]{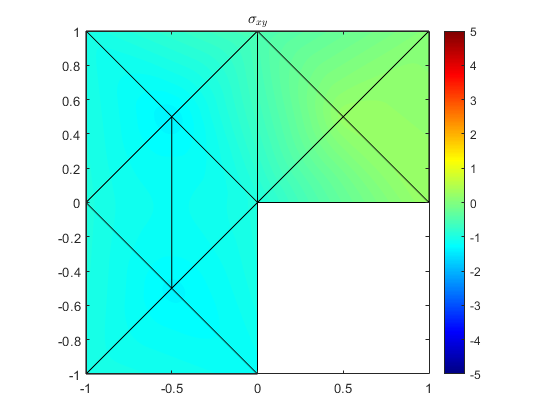}}
	\subfloat[$\sigma_{yy}$]{\includegraphics[trim=45 20 60 20, clip, width=0.31\linewidth]{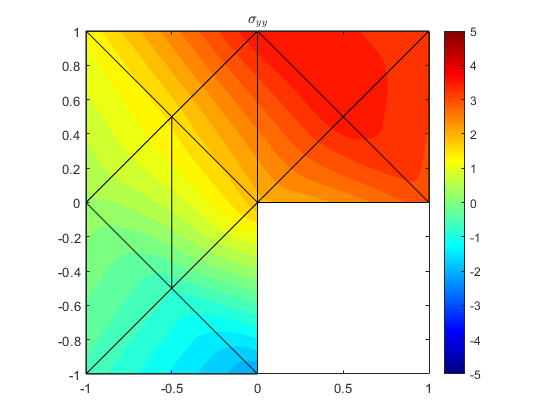}}
	
	\subfloat[$\sigma_{xx}$]{\includegraphics[trim=45 20 60 20, clip, width=0.31\linewidth]{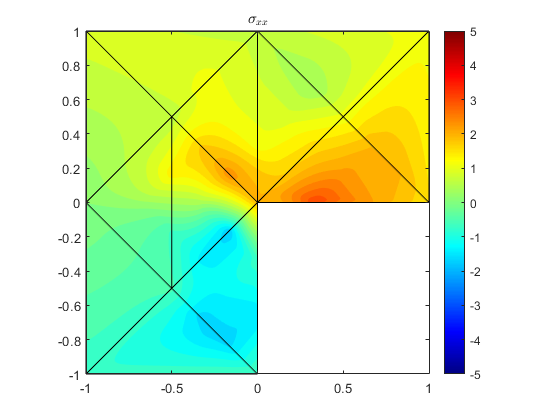}}
	\subfloat[$\sigma_{xy}$]{\includegraphics[trim=45 20 60 20, clip, width=0.31\linewidth]{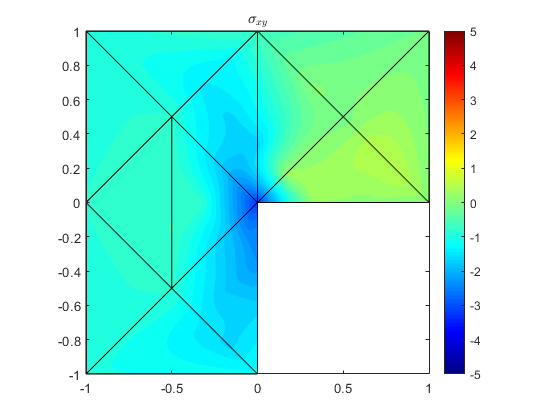}}
	\subfloat[$\sigma_{yy}$]{\includegraphics[trim=45 20 60 20, clip, width=0.31\linewidth]{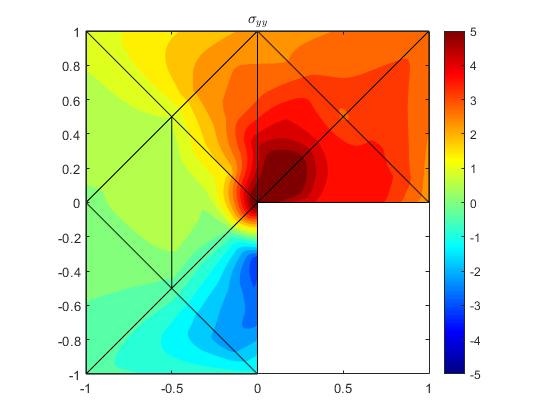}}
	
	\subfloat[$\sigma_{xx}$]{\includegraphics[trim=45 20 60 20, clip, width=0.31\linewidth]{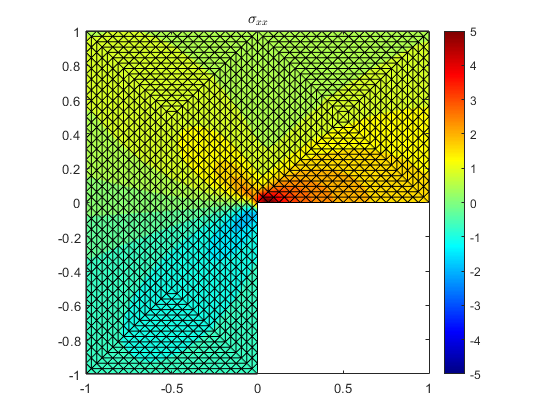}}
	\subfloat[$\sigma_{xy}$]{\includegraphics[trim=45 20 60 20, clip, width=0.31\linewidth]{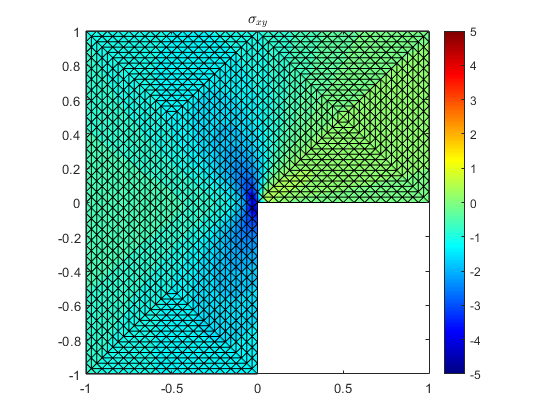}}
	\subfloat[$\sigma_{yy}$]{\includegraphics[trim=45 20 60 20, clip, width=0.31\linewidth]{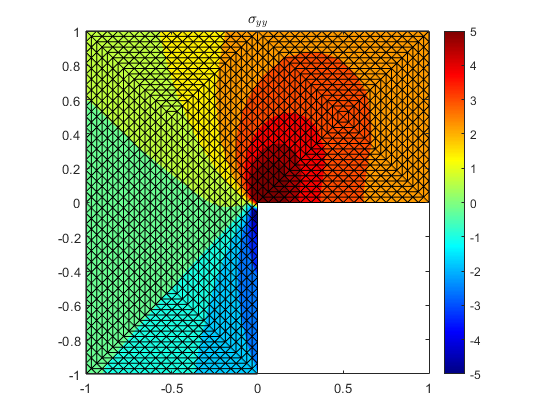}}
	\caption{Stress distributions for the $P_1$ Lagrange finite elment method on $\mathcal{T}_0$ (top), the coupling method $\HZ_3(\omega_1^-)+\CG_1(\Omega\setminus\omega_1^-)$ on $\mathcal{T}_0$ (middle), and the $P_1$ Lagrange finite element method on $\mathcal{T}_4$ (bottom) as reference.}
	\label{fig:Lshapesigma}
\end{figure}

\begin{figure}[htbp]
	\centering
	\includegraphics[trim=50 5 50 5, clip, width=0.63\textwidth]{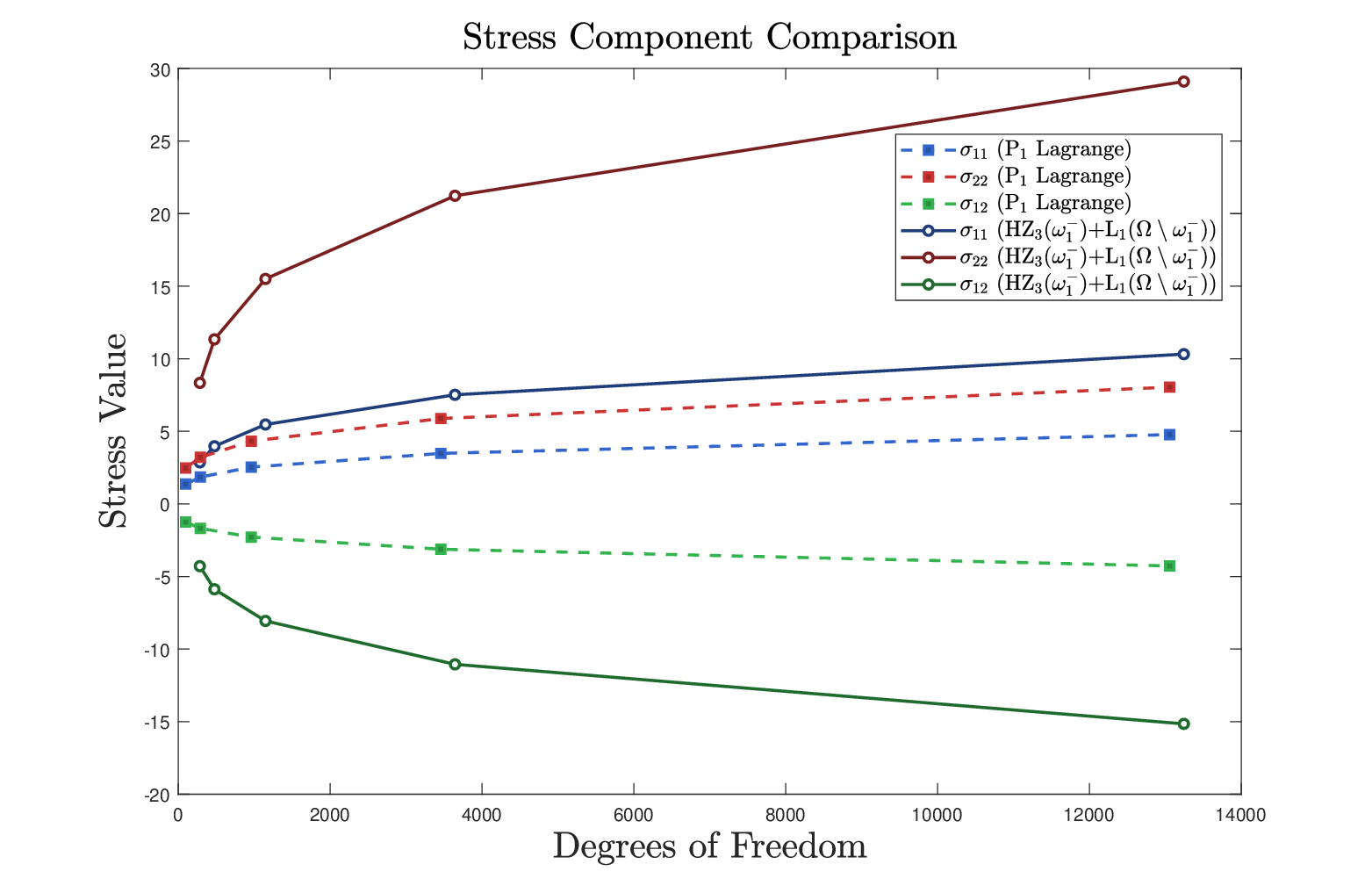}
	\caption{ Stress component comparison at point $(0,0)$.}
	\label{fig:OriStress}
\end{figure}

\subsubsection{Comparison of different coupling methods}

To assess the performance of different coupling strategies, we vary the layer number~$i$ and polynomial degrees $k$ and $m$ in the $\HZ_k(\omega_i^-) + \CG_m(\Omega \setminus \omega_i^-)$ method. Figure~\ref{fig:LshapeErr} compares the~$\HZ_3(\omega_i^-) + \CG_1(\Omega \setminus \omega_i^-)$ method (left) and $\HZ_k(\omega_1^-) + \CG_1(\Omega \setminus \omega_1^-)$ method (right). The numerical stress by all these methods converges at an approximate rate $0.54$, which coincides with the theoretical result. 
As the layer number $i$ increases, the error of stress by the~$\HZ_3(\omega_i^-)+\CG_1(\Omega\setminus\omega_i^-)$ method gradually decreases, and the error of the $P_1$ Lagrange finite element method is approximately three times larger than that of the $\HZ_3(\omega_5^-)+\CG_1(\Omega\setminus\omega_5^-)$ method on the mesh $\mathcal{T}_4$. This implies that a small region $\Omt$ of Hu--Zhang element can significantly improve stress accuracy.
In contrast, varying $k$ in $\HZ_k(\omega_1^-) + \CG_1(\Omega \setminus \omega_1^-)$ method has little impact due to the limited accuracy of the $P_1$ Lagrange approximation.

Furthermore, Table~\ref{Table:LResult1} reports the stress errors on
the finest mesh for the coupling method $\HZ_3(\omega_1^-)+\CG_m(\Omega\setminus\omega_1^-)$ and the $P_m$ Lagrange finite element method.
For a fixed $m$, the coupling method achieves consistently higher stress accuracy than the $P_m$ Lagrange element with comparable Dofs. %This implies that the coupling method can achieve a good balance between accuracy and computational efficiency, as it strategically employs Hu--Zhang and Lagrange elements to resolve regions with and without stress singularity, respectively.
This indicates that the coupling method provides an effective accuracy-efficiency trade-off by using Hu--Zhang elements in the stress concentration region and Lagrange elements elsewhere.

\begin{figure}[htbp]
	\centering
	\subfloat{\includegraphics[trim=10 0 30 0, clip, width=0.45\linewidth]{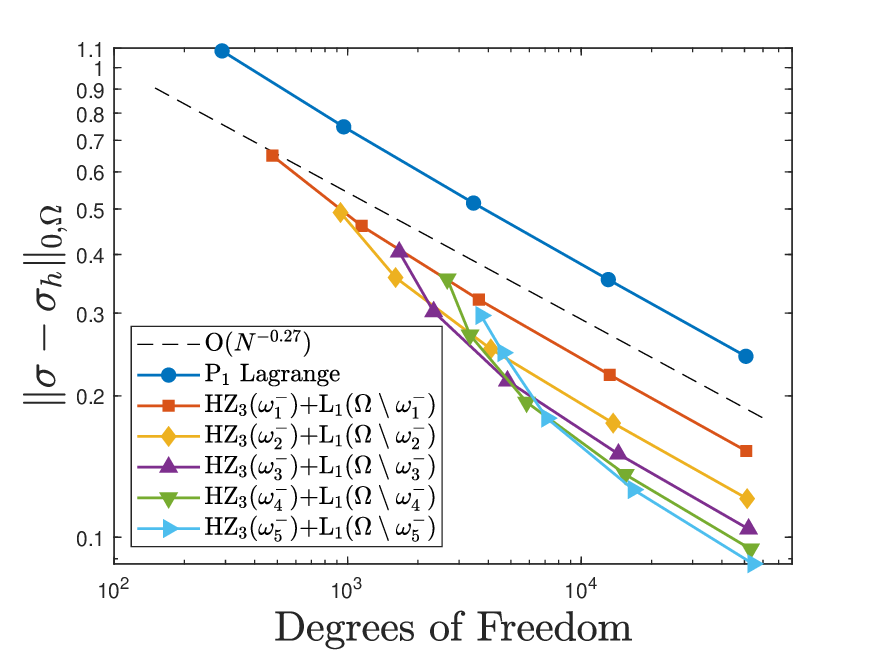}}
	\hspace{2.0em}
	\subfloat{\includegraphics[trim=10 0 30 0, clip, width=0.45\linewidth]{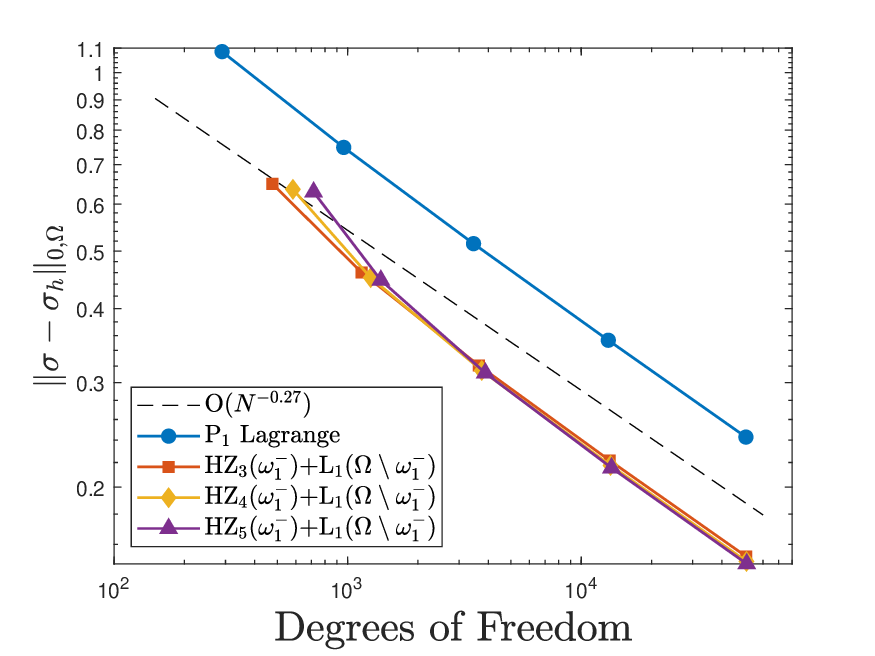}}
	\caption{Numerical comparison of stress errors for various methods.}
	\label{fig:LshapeErr}
\end{figure} 

\begin{table}[htbp]\centering
	\caption{Numerical results on the finest mesh obtained with the coupling method $\HZ_3(\omega_1^-)+\CG_m(\Omega\setminus\omega_1^-)$
	 and the $P_m$ Lagrange finite element method.}\label{Table:LResult1}
	\renewcommand{\arraystretch}{1.2} % 调整行高
	\begin{tabular}{c|ccccc}
		\hline
		$m$  & & $m=1$  & $m=2$  & $m=3$  & $m=4$  \\ \hline
		\multirow{2}{*}{$\HZ_3(\omega_1^-)+\CG_m(\Omega\setminus\omega_1^-)$} & Dofs                             & 13244  & 50860  & 113042 & 199790 \\
		& $\|\sigma-\sigma_h\|_{0,\Omega}$ & 0.2217 & 0.0834 & 0.0726 & 0.0723 \\ \hline
		\multirow{2}{*}{$P_m$ Lagrange}   & Dofs                             & 13058  & 50690  & 112898 & 199682 \\
		& $\|\sigma-\sigma_h\|_{0,\Omega}$ & 0.3537 & 0.1709 & 0.1189 & 0.0868 \\ \hline
	\end{tabular}
\end{table}
\subsection{Bimaterial Cook's membrane problem}
Consider a benchmark problem for a tapered plate composed of two materials with different mechanical properties~\cite{MR4056294,MR2574903}, which is an interface problem. The domain $\Omega$ occupied by the plate is the convex hull of four vertices $(0,0)$, $(48,44)$, $(48, 60)$ and $(0,44)$. A subdomain of~$\Omega$, which is the convex hull of vertices $(12,20.25)$, $(36,38.75)$, $(36, 50.25)$, is occupied by one material and denoted by $\Omo$, and the remaining domain is occupied by the other material and denoted by $\Omt$. Here, the boundary of $\Omo$ naturally forms the interface between the two materials.

The initial mesh is depicted in Figure~\ref{exam3:fig:cookuyres}. 
Although Assumption~\ref{Ass:kroncon} is not satisfied, the subsequent numerical results demonstrate that the proposed coupling method is still effective in handling such problems, as also observed in Example \ref{3DExam}. 
The subdomain $\Omo$ consists of a compressible and stiff material with  $E=250$ and $\nu=0.35$, whereas the nearly incompressible region $\Omt$ with $E=80$ and $\nu = 0.499999$. The plate is clamped at $x=0$ and loaded by a shear force $F=100$ on the right boundary, while the remaining boundaries are traction-free.
The inconsistent traction boundary conditions at the top-right and bottom-right corners lead to reduced solution regularity~\cite{MR4193452}. 
% In this section, the coupling~$P_3(\Omt)+P_k$ method denotes the coupling method of the $P_3$ Hu--Zhang element on $\Omt$ and the $P_k$ Lagrange element on $\Omo$.
%The initial three levels of triangulations are displayed in Figure \ref{exam3:fig:mesh}.
% where the subdomain $\Omo$ is represented in blue, $\Omt$ in red and the interface $\Gamma$ in black.
%\begin{figure}[htbp]
%    \centering
%    \subfloat{\includegraphics[trim=75 20 60 10, clip, width=0.25\linewidth]{BiCookMesh1.eps}} 
%    \caption{The initial mesh for the Cook’s membrane problem. Red region: Hu--Zhang subdomain $\Omt$; reminder: Lagrange subdomain $\Omo$.}
%    \label{exam3:fig:mesh}
%\end{figure}

%Due to inconsistent traction boundary conditions at the top-right corner of the plate, we adopt the partial $C^0$ vertex continuity relaxation method proposed in \cite{MR4193452} to modify the Hu--Zhang finite element space and ensure compatibility with prescribed boundary conditions.
% {\color{red}Furthermore, to address the reduced regularity of solution induced by the inconsistent traction boundary conditions, we employ the coupling $P_3$ Hu--Zhang and $P_k$ ($k=1,2,3$) Lagrange finite element method to solve the bimaterial Cook's membrane problem.}

As shown in Figure~\ref{exam3:fig:cookuyres}, the vertical displacement at the top-right corner of the plate by the coupling methods~$\HZ_3(\Omt)+\CG_m(\Omo)$ with $m=2$, $3$ and $4$ match well with the reference solution even on the coarsest mesh, while it  takes four levels of mesh refinements for the result by that with $m=1$ to converge.
For this benchmark problem with solution of low regularity, the~$\HZ_3(\Omt)+\CG_m(\Omo)$ method remains reliable for moderate $m$ but deteriorates for very low-order Lagrange discretizations. Figure~\ref{exam3:fig:cookres} shows the numerical solution obtained by $\HZ_3(\Omt) + \CG_2(\Omo)$, which captures the singular behavior and agrees with~\cite{MR4056294}.

% The evolution of the vertical displacement of the plate's top right corner is displayed in Figure \ref{exam3:fig:cookuyres} as a function of the number of mesh elements in the domain $\Omega$.
% The numerical solution obtained by the $P_3$ Hu--Zhang finite element method on the fifth refinement level is established as the reference solution.
% Notably, the coupling $P_3$ Hu--Zhang and $P_2/P_3$ Lagrange method achieves immediate agreement with the reference solution even on the coarsest mesh. In contrast, the cuopling $P_3$ Hu--Zhang and $P_1$ Lagrange method exhibits significant deviations on the initial mesh, but demonstrates progressive convergence toward the reference solution under mesh refinement.
% These behaviors confirms the coupling method's effectiveness in mitigating locking phenomena.
\begin{figure}[htbp]
	\centering
	\includegraphics[trim=75 20 60 10, clip, width=0.28\linewidth]{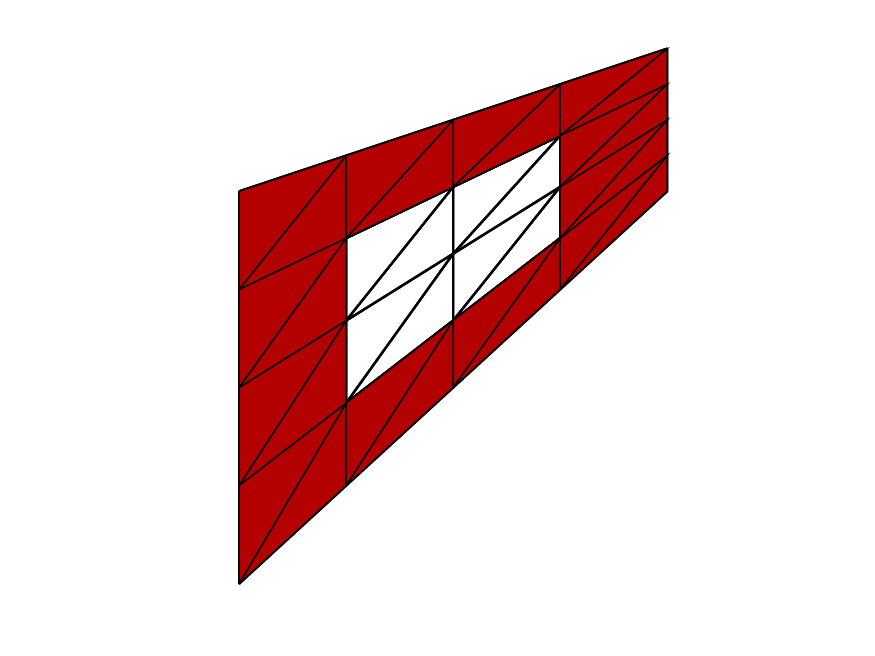}
    \includegraphics[trim=5 5 5 5, clip, width=0.55\linewidth]{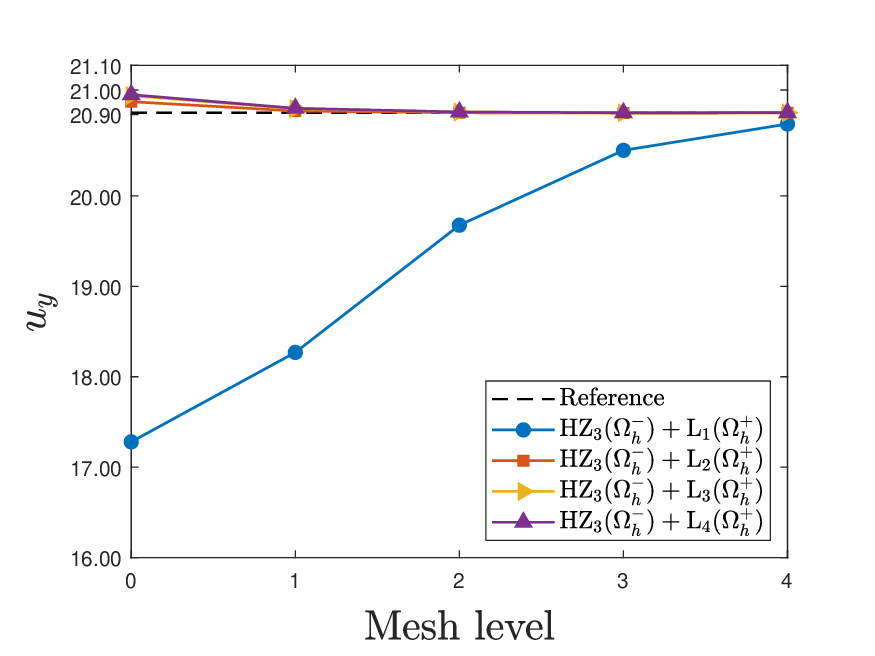}
	\caption{The initial mesh for Cook’s membrane problem (left), and vertical displacement at plate's top-right corner vs. the level of mesh refinement (right), where the solution by the $P_3$ Hu--Zhang finite element method on the fourth-level mesh is taken as the reference solution.}
	\label{exam3:fig:cookuyres}
\end{figure}
\begin{figure}[htbp]
    \centering
    \subfloat[$u_x$]{\includegraphics[trim=75 20 60 10, clip, width=0.25\linewidth]{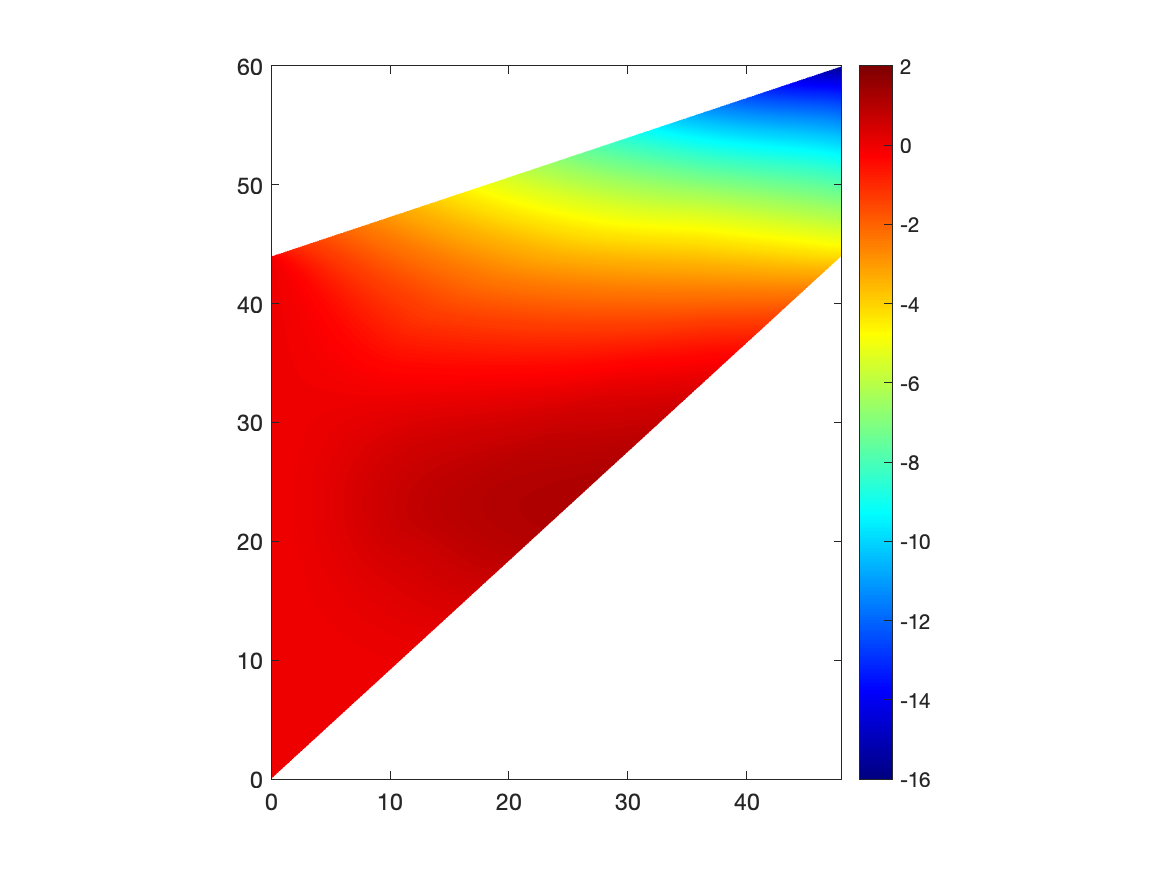}}
    \subfloat[$u_y$]{\includegraphics[trim=75 20 60 10, clip, width=0.25\linewidth]{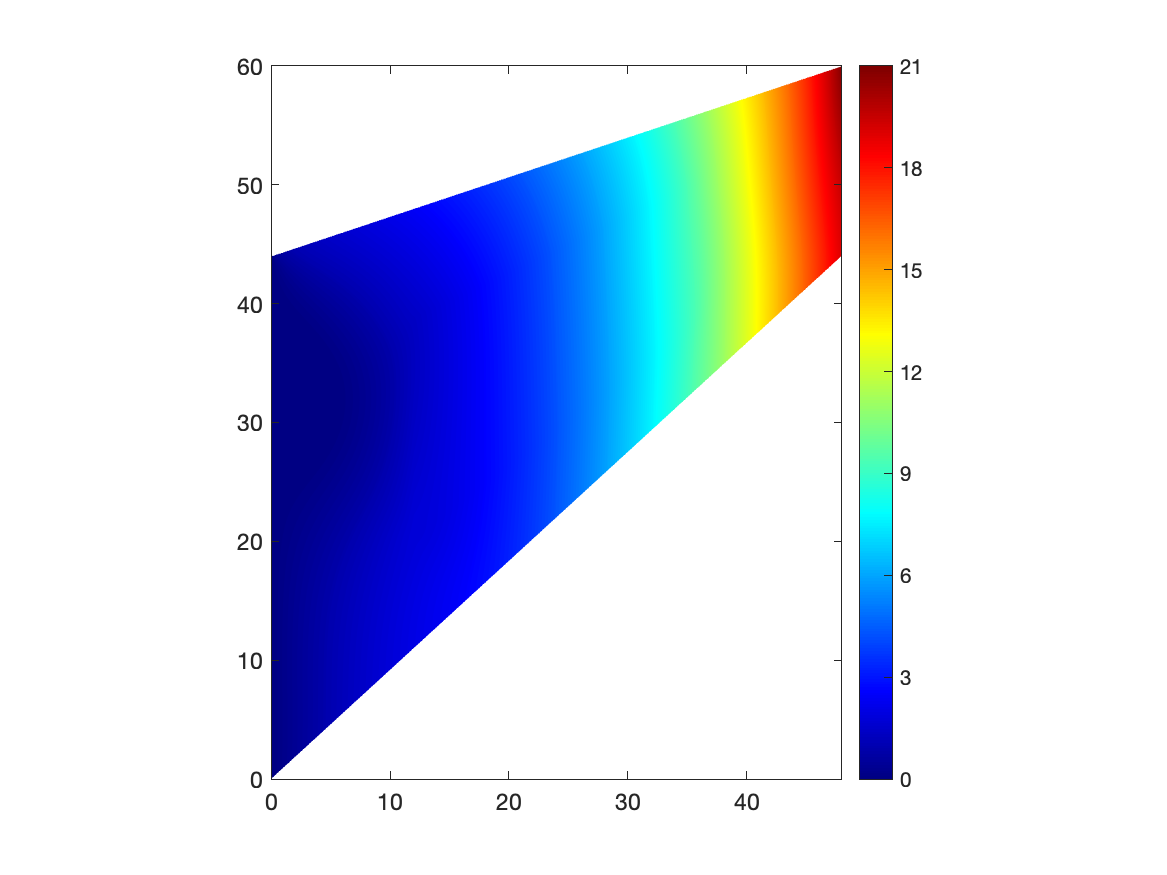}}
    
    \subfloat[$\sigma_{xx}$]{\includegraphics[trim=75 20 60 10, clip, width=0.25\linewidth]{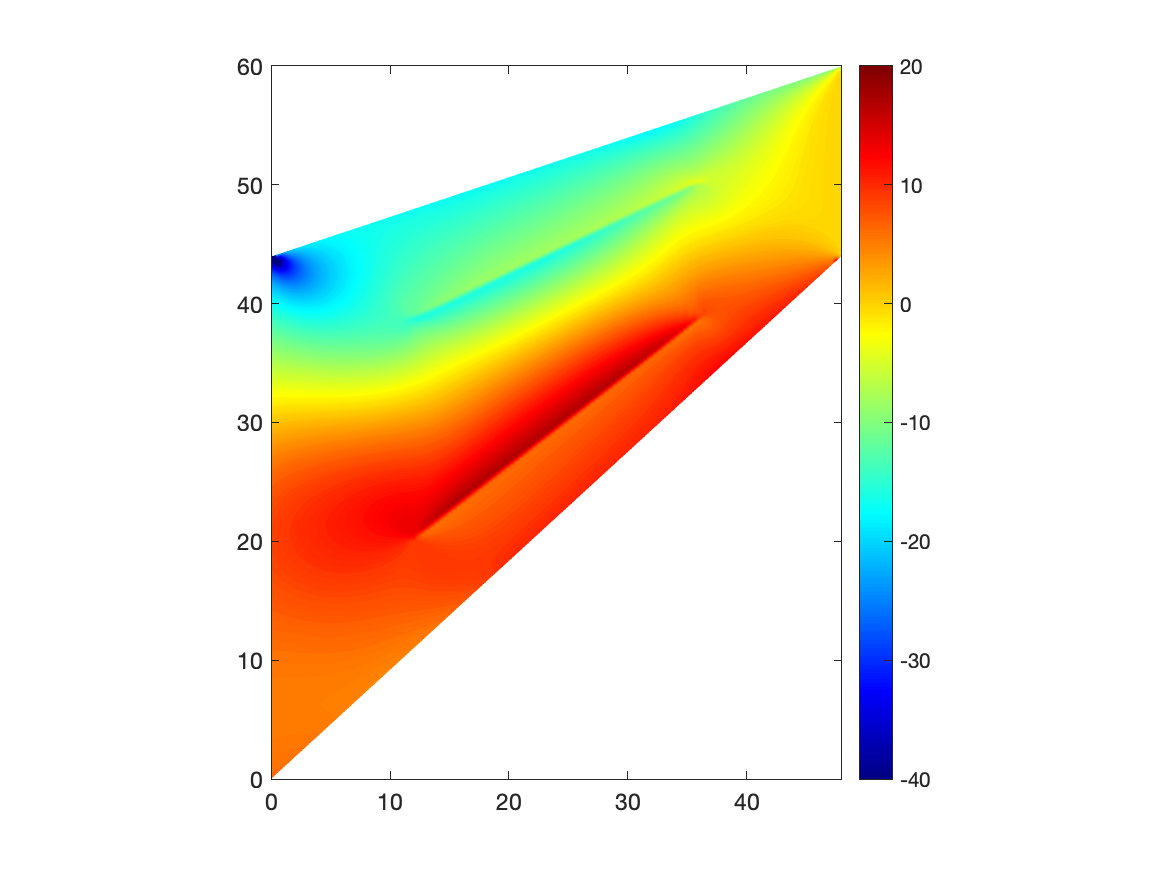}}
     \subfloat[$\sigma_{xy}$]{\includegraphics[trim=75 20 60 10, clip, width=0.25\linewidth]{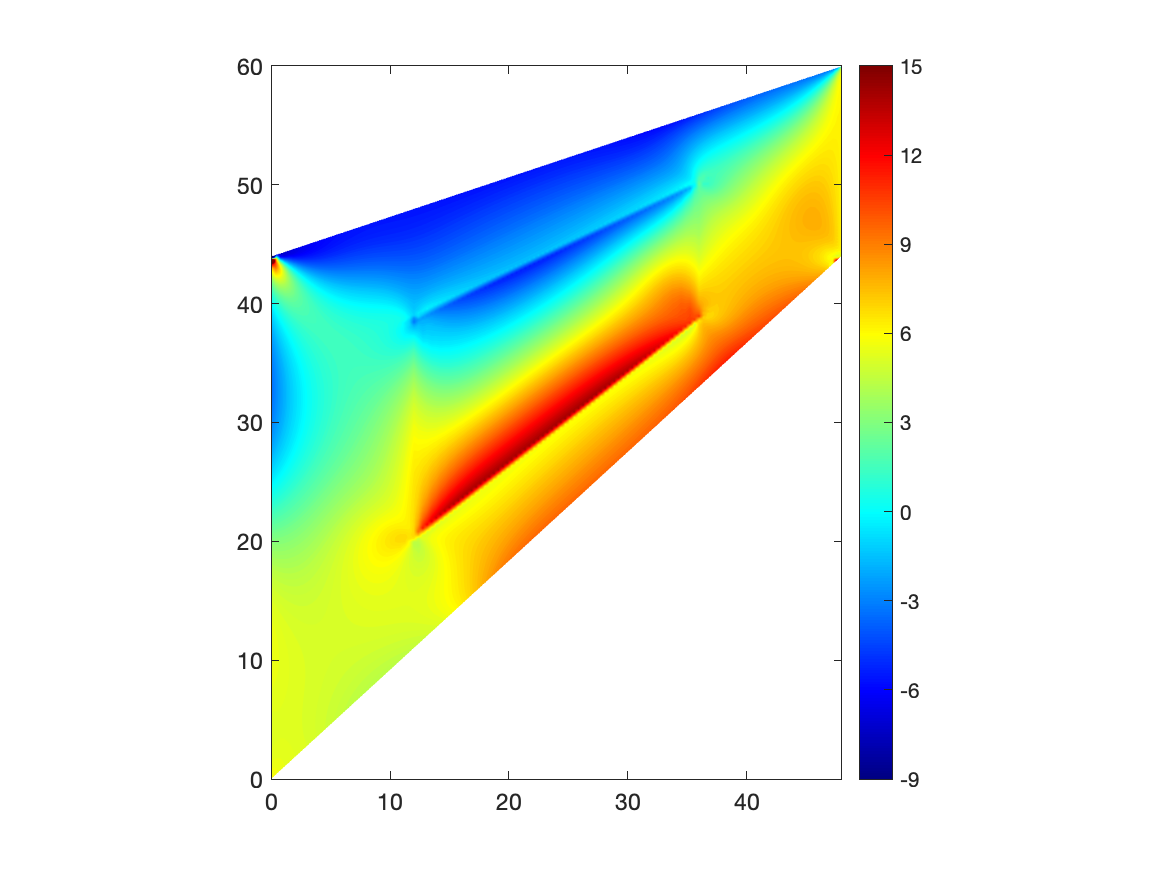}}
    \subfloat[$\sigma_{yy}$]{\includegraphics[trim=75 20 60 10, clip, width=0.25\linewidth]{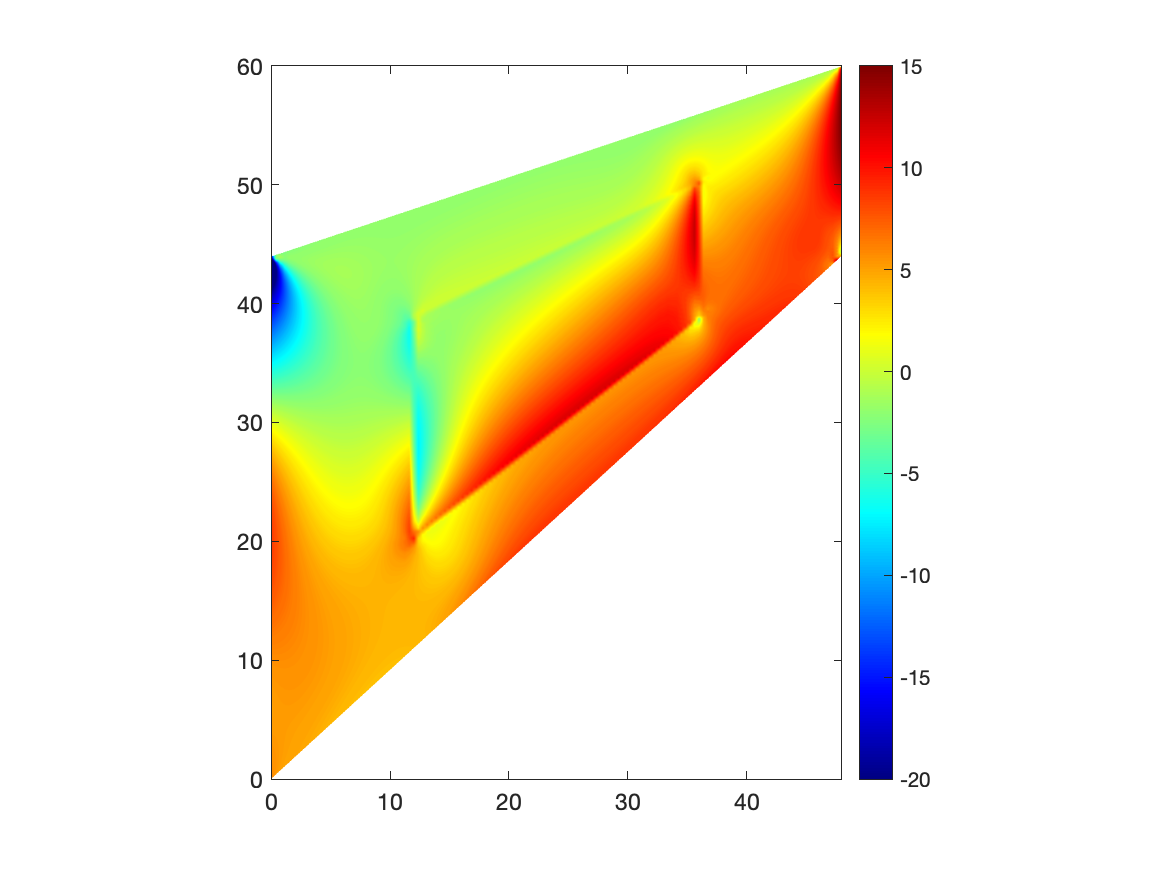}}
    \caption{Results of the coupling $\HZ_3(\Omt)+\CG_2(\Omo)$ method for bimaterial Cook's membrane problem on the fourth-level refined mesh.}
    % \caption{Results of the $P_3$ Hu--Zhang and $P_4$ Lagrange hybrid finite element method for bimaterial Cook's membrane problem in the nearly incompressible case of $\nu=0.4999$ on the fifth level of mesh refinement.}
    \label{exam3:fig:cookres}
\end{figure}

%%%%%%%%%%%%%%%%%%%%%%%%%%%%%%%%%%%%%%%%%%%%%%%%%%%%%%%%%%%%%%%%%%%%%%%%%%%%%%%%%%%%%%%%%%%%%%%%%%%%%%%%%%%%%%%%%%%%%%%%%%%%%%%%%%%%%%%%%%%%%%%%%%%%%%%%%%%%%%%%%%%%%%%%%%%%%%%%%%%%%%%%%%%%%%%%%%%%%%%%%%%%%%%%%%%%%%%%%%%%%%%%%%%%%%%%%%%
\subsection{Kirsch’s plate problem}
Consider a standard benchmark problem consisting of a plate with a central circular hole subjected to a uniform horizontal traction of magnitude $T_x$. By exploiting the symmetry of the geometry and loading, the problem is modeled on a truncated quarter plate. The exact solution, evaluated at the boundary of the finite quarter plate, is given here for reference \cite{hughes2005isogeometric}:
\begin{align*}
	\sigma_{rr}(r,\theta) &= \frac{T_x}{2}\left(1-\frac{R^2}{r^2}\right)+\frac{T_x}{2}\left(1-4\frac{R^2}{r^2}+3\frac{R^4}{r^4}\right)\cos 2\theta,\\
	\sigma_{\theta\theta}(r,\theta) &= \frac{T_x}{2}\left(1+\frac{R^2}{r^2}\right)-\frac{T_x}{2}\left(1+3\frac{R^4}{r^4}\right)\cos 2\theta,\\
	\sigma_{r\theta}(r,\theta)&=-\frac{T_x}{2}\left(1+2\frac{R^2}{r^2}-3\frac{R^4}{r^4}\right)\sin 2\theta.
\end{align*}
Specifically, the body force is $f=(0,0)^T$. Symmetric boundary conditions $u\cdot n = 0$ and $(\sigma n)\cdot \tau = 0$ are imposed on the left and bottom boundaries. A homogeneous Neumann condition $\sigma n = 0$ is prescribed on the circular boundary, while exact Neumann data are imposed on the top and right boundaries.
The detailed setup is illustrated in Figure.~\ref{Fig:Kirch}. In this experiment, a uniaxial tensile load $T_x = 10$ is applied, yielding a stress concentration $\sigma_{xx} = 30$ at point $(r,\theta) = (R, \pi/2)$, which is the concentration set in this case. Since the domain $\Omega$ has a curved boundary, we use isoparametric Lagrange elements~\cite{bernner2008the} and curved Hu--Zhang elements~\cite{chen2025huzhangelementlinearelasticity}, which are still denoted by $P_m$  Lagrange element, $P_k$ Hu--Zhang element, and the coupling method $\HZ_k(\omega_i^-) + \CG_m(\Omega \setminus \omega_i^-)$ for simplicity.

\begin{figure}[htbp]
	\centering
	\begin{tikzpicture}[scale=1.35,
		dashedline/.style={dashed, thick},
		]
		% 定义正六边形的六个顶点
		\coordinate (O) at (0,0);
		\coordinate (A) at (1,0);
		\coordinate (B) at (4,0);
		\coordinate (C) at (4,4);
		\coordinate (D) at (0,4);
		\coordinate (E) at (0,1);
		
		\draw[thick] (A) -- (B) -- (C) -- (D) -- (E) ;
		\draw[thick] (A) arc (0:90:1);
		\draw [dashed,thick] (O)--(A);
		\draw [dashed,thick] (O)--(E);
		\draw [->,thick] (O)--(45:1); 
		
		\draw[thick] (-0.4,0) -- (-0.1,0);  
		\draw[thick] (-0.4,4) -- (-0.1,4); 
		\draw [<->,thick] (-0.25,0)--(-0.25,4); 
		
		\draw[thick] (0,-0.4) -- (0,-0.1);  
		\draw[thick] (4,-0.4) -- (4,-0.1); % 底部直线
		\draw [<->,thick] (0,-0.25)--(4,-0.25); 
		
		\node[anchor=north west] at (1.2, 3.2) {Material properties:};
		\node[anchor=north west] at (1.2, 2.8) {$R = 1$};
		\node[anchor=north west] at (1.2, 2.5) {$L = 4$};
		\node[anchor=north west] at (1.2, 2.2) {$E = 10^5$};
		\node[anchor=north west] at (1.2, 1.9) {$\nu = 0.3$};
		
		\node [left] at (-0.4,2) {L};
		\node [below] at (2,-0.4) {L};
		
		\node[right]  at (0.7,0.74){$\sigma n = 0$};
		\node [above left] at (0.45,0.4) {$R$};
		\node [rotate = 90, right] at (0.2,1.8){Symmetry};
		\node [above] at (2.5,0){Symmetry};
		\node [below] at (2,4) {Exact traction};
		\node [rotate = -90, left] at (3.8,1.5){Exact traction};
		
		\draw [->,thick] (4.2,2.2) -- (4.8,2.2);
		\draw [->,thick] (4.2,2) -- (4.8,2);
		\draw [->,thick] (4.2,1.8) -- (4.8,1.8);
		\node [above] at (4.5,2.2){$T_x$};
	\end{tikzpicture}
	\caption{Kirsch's plate problem.}
	\label{Fig:Kirch}
\end{figure}
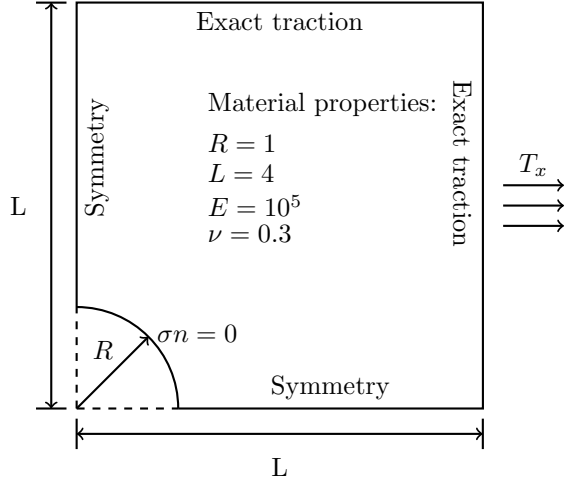

Figure~\ref{fig:Krichsigmaxx} compares the stress distribution by the $P_1$ Lagrange element and the $\HZ_3(\omega_1^-)+\CG_1(\Omega\setminus\omega_1^-)$ method on the initial mesh, which indicates a more accurate numerical behavior of the coupling method. The numerical values of $\sigma_{xx}$ at the stress concentration point by the $P_m$ Lagrange element, the $\HZ_3(\omega_i^-)+\CG_1(\Omega\setminus\omega_i^-)$, and the $\HZ_3(\omega_1^-)+\CG_m(\Omega\setminus\omega_1^-)$ methods with $1\le i\le 5$ and $1\le m\le 3$ are reported in Figure~\ref{fig:KrichsResult}.
The left subfigure shows that the~$\HZ_3(\omega_i^-)+\CG_1(\Omega\setminus\omega_i^-)$ method achieves significantly higher accuracy than the $P_1$ Lagrange element, even when the Hu--Zhang element is employed only on a small region. 
In particular, the solution by the~$\HZ_3(\omega_i^-)+\CG_1(\Omega\setminus\omega_i^-)$ method with $i\ge 2$ closely matches with the reference solution even on the initial mesh. 
The right subfigure indicates that, for a fixed $m$, the~$\HZ_3(\omega_1^-)+\CG_m(\Omega\setminus\omega_1^-)$ method consistently outperforms the $P_m$ Lagrange element. This advantage is especially pronounced for smaller values of $m$.
	
\begin{figure}[htbp!]
	\centering
	\subfloat[$\sigma_{xx}$]{\includegraphics[trim=20 20 20 20, clip, width=0.48\linewidth]{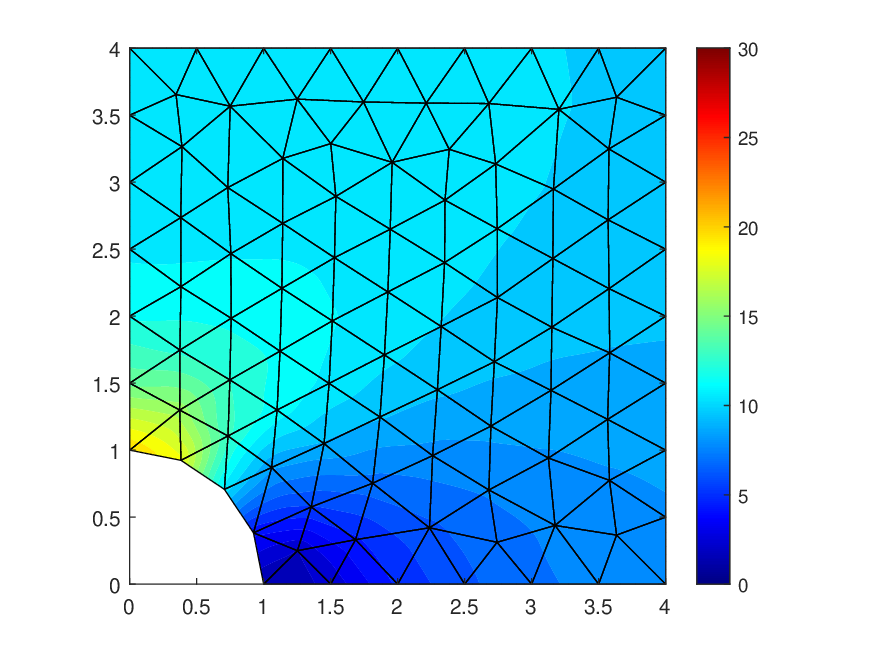}}
	\subfloat[$\sigma_{xx}$]{\includegraphics[trim=20 20 20 20, clip, width=0.48\linewidth]{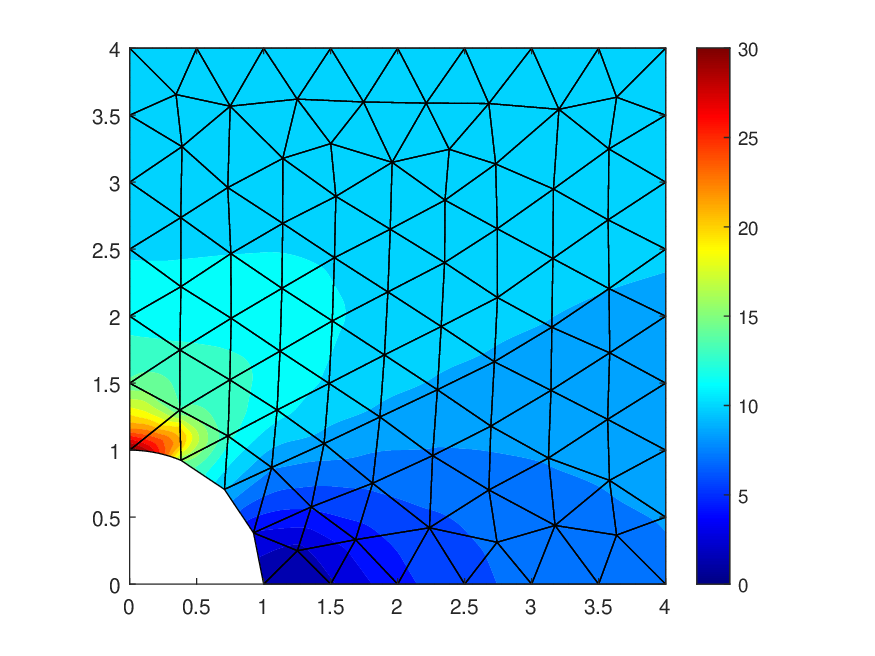}}
	\caption{Stress distributions for the $P_1$ Lagrange element (left) and the coupling  method $\HZ_3(\omega_1^-)+\CG_1(\Omega\setminus\omega_1^-)$ (right) on the initial mesh.}
	\label{fig:Krichsigmaxx}
\end{figure}
\begin{figure}[htbp!]
	\centering
	\subfloat{\includegraphics[trim=10 0 30 0, clip, width=0.45\linewidth]{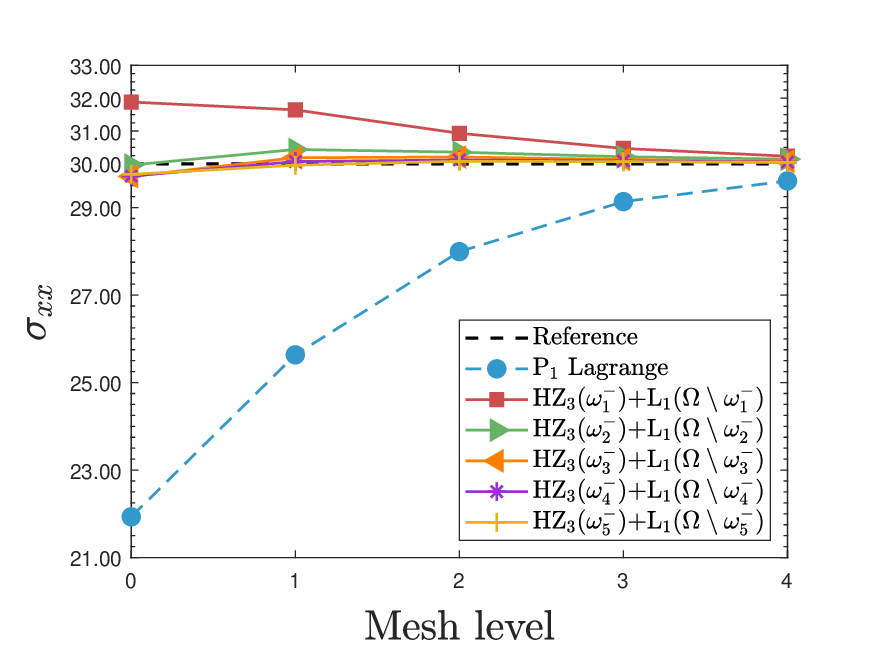}}
	\hspace{2.0em}
	\subfloat{\includegraphics[trim=10 0 30 0, clip, width=0.45\linewidth]{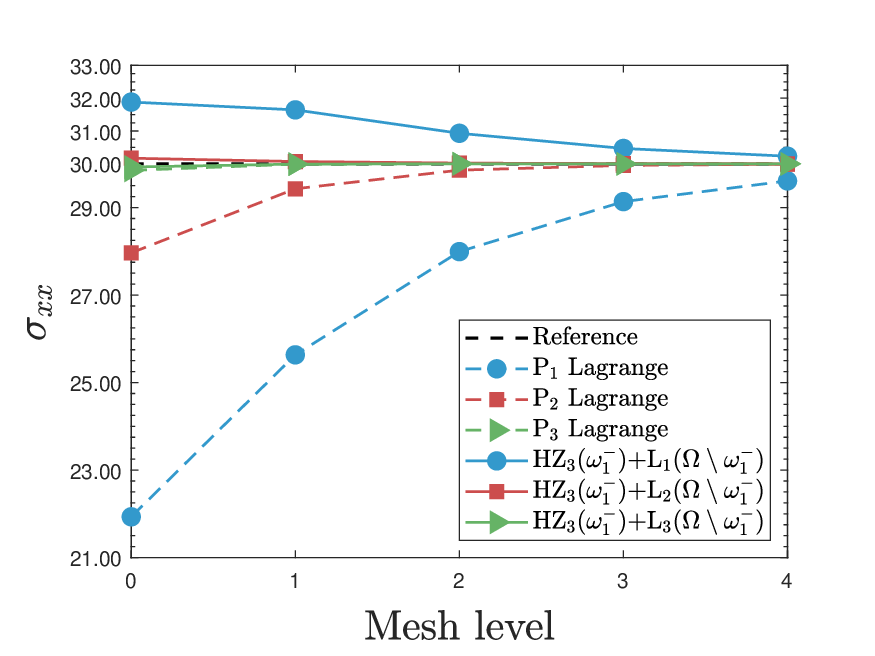}}
	\caption{Numerical comparison of $\sigma_{xx}$ at the stress concentration point for various methods.}
	\label{fig:KrichsResult}
\end{figure} 
\subsection{The Girkmann problem}
\label{girk-mix} 

The Girkmann problem is a classical benchmark in structural mechanics that describes an axisymmetric thin shell supported by stiffening rings, whose cross-sectional geometry is illustrated in Figure~\ref{chap3:fig:1}.
Let $\Omega =  \Oms \cup \Omr$ denote the shell–ring system with interface $\Gamma_\alpha = \partial\Oms\cap\partial\Omr$, and $\Gamma_{AB}$ be the bottom segment of the ring with dimensions $a = 0.60$~m and $b = 0.50$~m. The material parameters are $E = 20.59$~GPa and $\nu = 0$. The radius of the middle surface of the shell is $R_m = \frac{R_c}{\sin\alpha}\approx 23.34$ m with meridional angle $\alpha = 40^\circ$, and the radius of the shell $R_c = 15$ m. The thickness of the shell is $h = 0.06$ m, and the slenderness ratio is $t = \frac{h}{R_m} \approx 2.57 \times 10^{-3}$. The shell is subjected to a downward gravitational force, while the bottom of the ring experiences an upward pressure:
$$
f=\begin{cases}
(0,-F)^T,& \mbox{on}~\Oms,\\
(0,0)^T,& \mbox{on}~\Omr,
\end{cases}
\qquad
g_n=\begin{cases}
(0,p)^T,&\mbox{on}~\Gamma_{AB},\\
(0,0)^T,&\mbox{on}~\partial\Omega\setminus\Gamma_{AB},
\end{cases}
$$
where $F=32.690$ kN/m and $p=27.256$ kPa. The geometric configuration at the interface $\Gamma_\alpha$ leads to the singular behavior of the elastic solution. This problem, originally introduced in \cite{MR0090249} with a detailed analysis of its classical solution later presented in \cite{MR2927141}, is a standard benchmark for evaluating the performance of finite element methods. The quantities of interest include the shear force $Q$ and the bending moment $M$ at the interface $\Gamma_\alpha$.
For more details, we refer readers to~\cite{Barna2010The} and the references therein.

% The results by both $hp$-version axisymmetric solid-extraction method in~\cite{Antti2012} and axisymmetric-DG-FEM method in~\cite{philippe2013} are $Q=943.65$ N/m and $M=-36.79$ Nm/m.
% For more details, we refer readers to~\cite{Barna2010The} and the references therein.

In this example, the Girkmann problem is modeled using the axisymmetric solid model~\cite{Antti2012}. Isoparametric Lagrange finite element~\cite{bernner2008the} and the curved Hu--Zhang finite element~\cite{chen2025huzhangelementlinearelasticity} are employed to handle the curved boundary of $\Oms$. Correspondingly, the coupling method $\HZ_{k-1}(\omega_1^-)+\CG_k(\Omega\setminus\omega_1^-)$ is defined as before, where the stress concentration set is the line segment $\Gamma_\alpha$.
It should be noted that, the stress solution obtained from the Hu--Zhang element in the coupling method is continuous, allowing direct computation of $Q$ and  $M$ by numerical integration. 
In contrast, for the $P_k$ Lagrange finite element method, these quantities are computed separately on each side of the interface $\Gamma_\alpha$ and then averaged.
%: on the shell side (denoted $Q_s$ and $M_s$) and on the ring side (denoted $Q_r$ and $M_r$), by direct numerical integration of the element-wise stresses. The final values are then obtained by averaging the values from both sides, i.e., 
%$$Q=\frac{Q_s+Q_r}{2} \quad \text{and} \quad M=\frac{M_s+M_r}{2}.$$

\begin{figure}[htbp]
\centering
\includegraphics[width=0.65\textwidth]{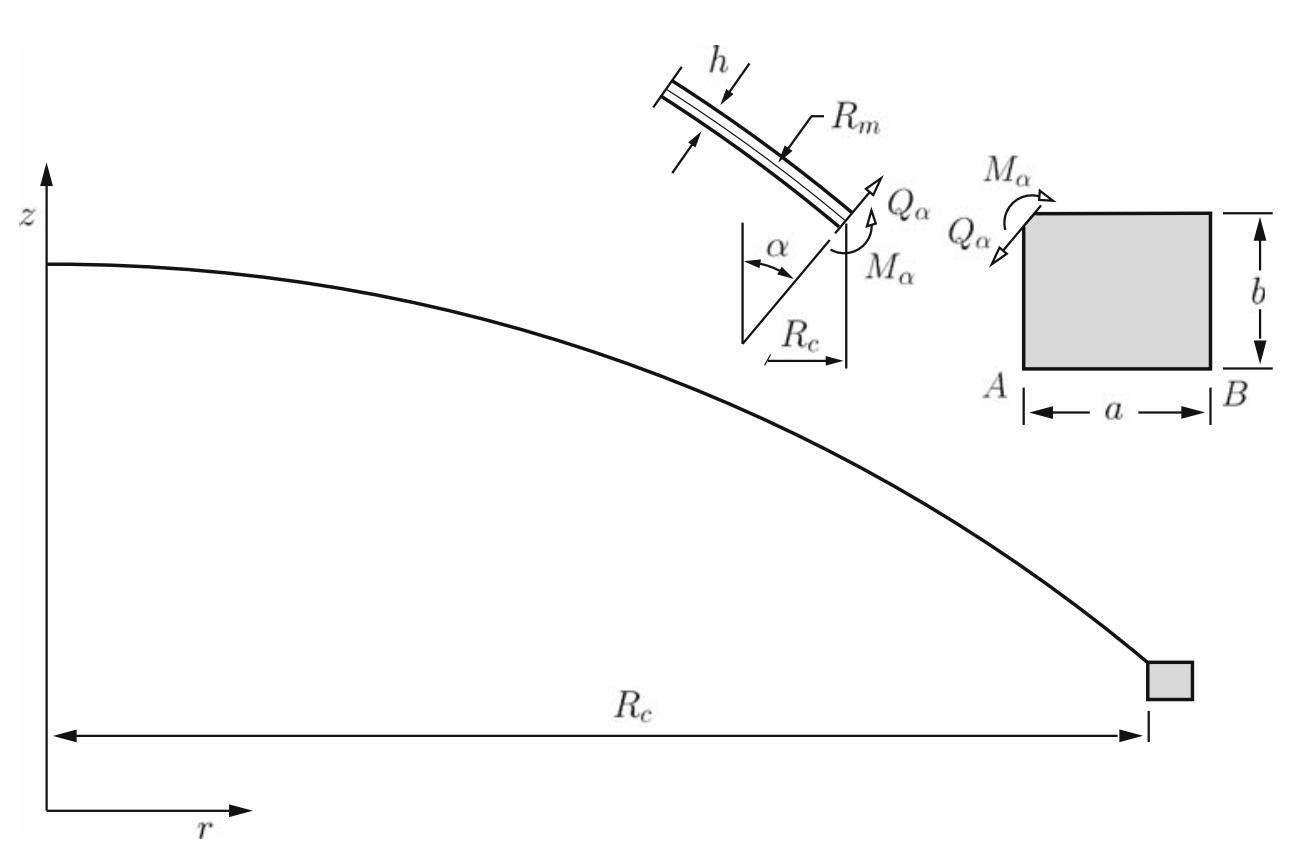}
	\caption{ The cross-sectional diagram of the Girkmann problem model.}
	\label{chap3:fig:1}
\end{figure}

\subsubsection{Meshes with high-aspect-ratio elements}
Consider a triangulation with high-aspect-ratio elements, as depicted in
Figure~\ref{fig:Low_quality_mesh}. This mesh employs a baseline size of $h=10$ m,
with local refinement to $h=0.01$ m at the interface $\Gamma_\alpha$. Consequently,
the first layer of elements on the right-hand side of $\Gamma_\alpha$ exhibits
geometric distortion, with aspect ratios (AR) exceeding $20$.

Compared with the reference values of $Q=943.65$~N/m and $M=-36.79$~Nm/m reported by both the $hp$-version axisymmetric solid-extraction method~\cite{Antti2012} and the axisymmetric-DG-FEM method~\cite{philippe2013}, the results in
Table~\ref{lowqualitymesh} demonstrate that the $\HZ_{k-1}(\omega_1^-)+\CG_k(\Omega\setminus\omega_1^-)$ method provides
stable and convergent approximations for various $k$, whereas the $P_k$ Lagrange element fails to converge even for $k=9$.
Furthermore, the $\HZ_{3}(\omega_1^-)+\CG_4(\Omega\setminus\omega_1^-)$ method with $2203$ Dofs significantly outperforms the $P_9$ Lagrange element with $8246$ Dofs, highlighting
the superior robustness and accuracy of the proposed coupling strategy.

% In this numerical example, we employ mesh size $h=10m$ and local mesh size $h=0.1m$ at the interface $\Gamma_\alpha$, see Figure \ref{fig:Low_quality_mesh}. The first right layer of mesh adjacent to the interface $\Gamma_\alpha$ exhibits severe distortion, with element aspect ratios exceeding $20:1$.
% We first use $P_k$ Lagrange element to compute $M$ and $Q$.
% Next, we replace the Lagrange elements in the first layers adjacent to both sides of $\Gamma_\alpha$ with Hu--Zhang elements and computed the values of $M$ and $Q$ based on the results from the Hu--Zhang formulation.
% As shown in Table \ref{lowqualitymesh}, the computational results indicate that the accuracy of the Lagrange elements is severely impacted by the low-quality mesh. However, the Hu--Zhang elements, due to their direct stress computation capability, maintain robust performance and deliver reliable results even under such mesh distortion conditions.

% \begin{figure}
% \centering
% \begin{minipage}{0.25\textwidth}
%     \centering
%     \includegraphics[width=1.0\linewidth]{low_quality_mesh.eps}
% \end{minipage}
% \hfill
% \begin{minipage}{0.25\textwidth}
%     \centering
%     \includegraphics[width=1.0\linewidth]{low_quality_mesh_fangda.eps}
% \end{minipage}
% % \caption{The low quality mesh and it's zoom at junction.}
% \caption{Baseline mesh with high-aspect-ratio elements (AR $\geq$ 20) and Zoomed view of distorted quadrilateral elements at shell-ring junction $\Gamma_\alpha$.}
% \label{fig:Low_quality_mesh}
% \end{figure}
\begin{figure}[htbp]
    \centering
    \subfloat{\includegraphics[trim=50 50 50 50, clip, width=0.25\linewidth]{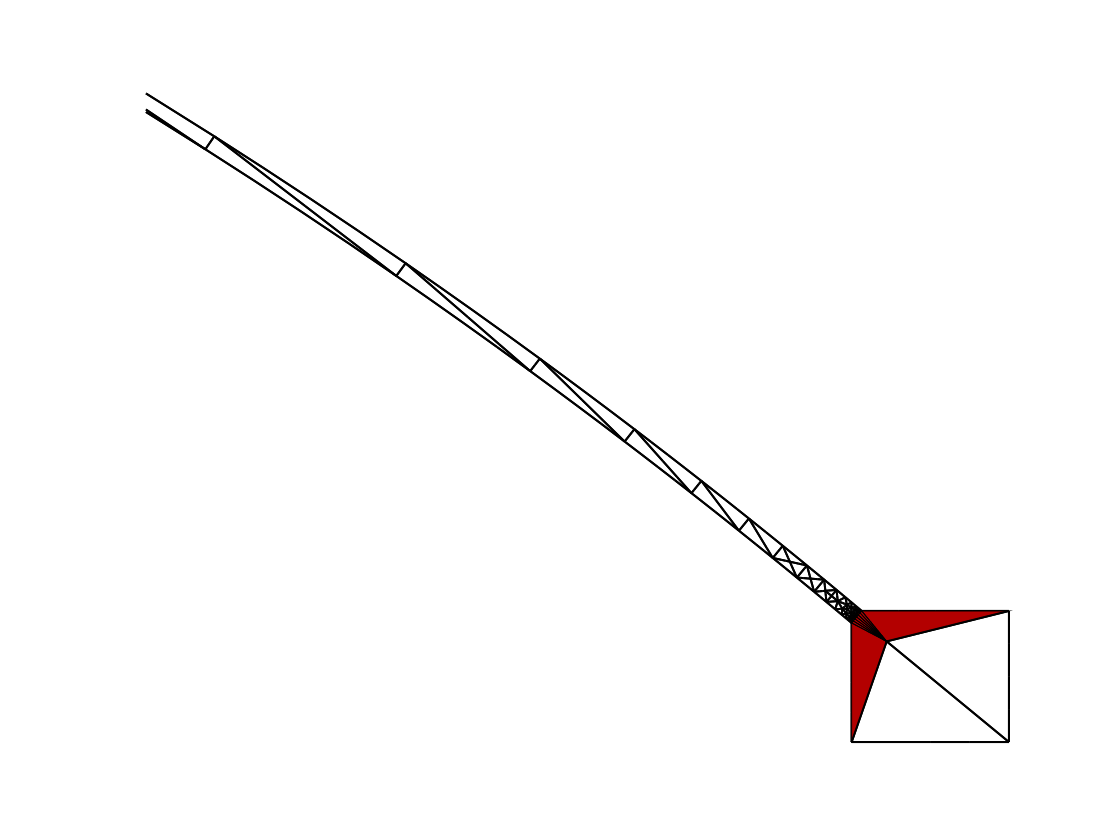}}
    \hspace{2.0em}
    \subfloat{\includegraphics[trim=50 50 50 50, clip, width=0.25\linewidth]{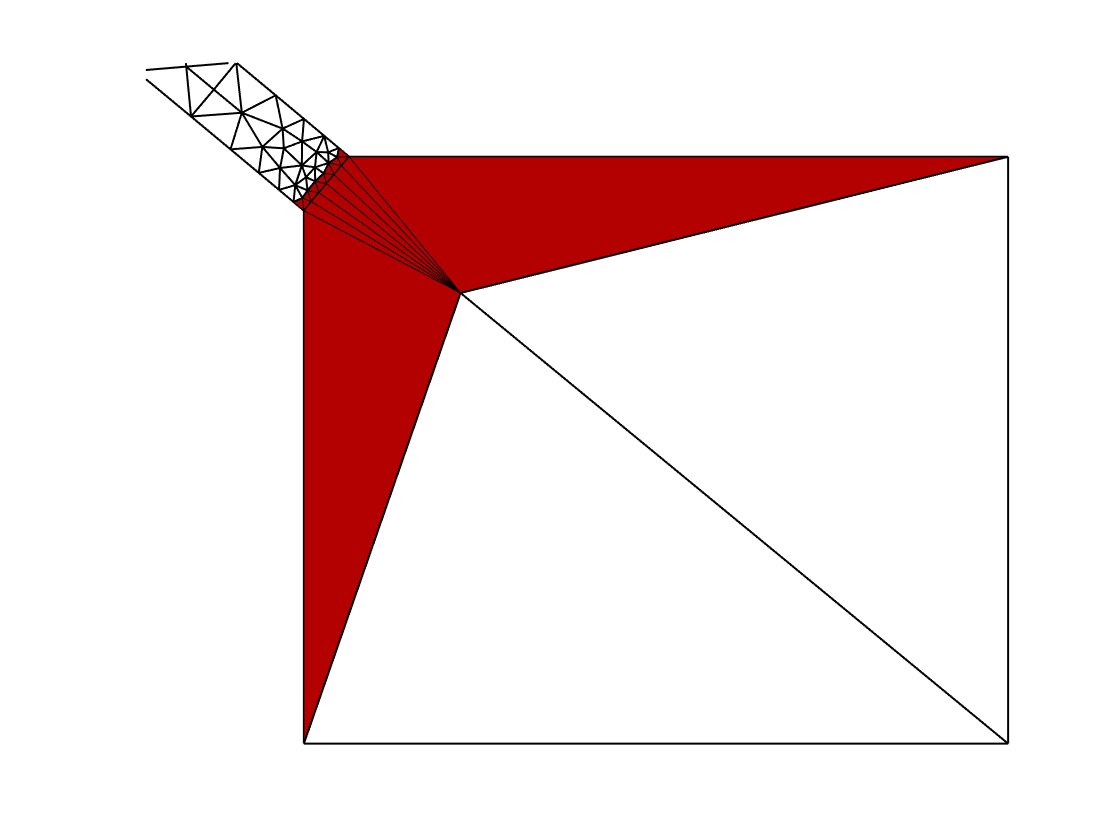}}
    \caption{Baseline mesh with high-aspect-ratio elements (AR $\geq$ 20) and zoomed view of distorted triangle elements at shell-ring junction $\Gamma_\alpha$.}
    \label{fig:Low_quality_mesh}
\end{figure}

\begin{table}[htp!]
\centering
\caption{
% \centering
Numerical results of the $P_k$ Lagrange element and the coupling $\HZ_{k-1}(\omega_1^-)+\CG_k(\Omega\setminus\omega_1^-)$ method on mesh with high-aspect-ratio elements.
}
\label{lowqualitymesh}
\begin{tabular}{c@{\hspace{0.5cm}}|c@{\hspace{0.5cm}}c@{\hspace{0.5cm}}c@{\hspace{0.5cm}}|c@{\hspace{0.5cm}}c@{\hspace{0.5cm}}c}
\hline
\hline
\multirow{2}{*}{$k$}&\multicolumn{3}{c}{$P_k$ Lagrange}&\multicolumn{3}{|c}{$\HZ_{k-1}(\omega_1^-)+\CG_{k}(\Omega\setminus\omega_1^-)$}\\
& Dofs &Q(N/m)& M(Nm/m)&Dofs& Q(N/m)&M(Nm/m) \\
\hline
4 & 1726 & 583.070 & -24.544 & 2203 & 943.478 & -36.172 \\
5 & 2642 & 703.836 & -26.741 & 3452 & 943.560 & -36.521 \\
6 & 3752 & 775.443 & -28.280 & 4979 & 943.597 & -36.641 \\
7 & 5056 & 814.391 & -29.464 & 6784 & 943.611 & -36.702 \\
8 & 6554 & 837.121 & -30.413 & 8867 & 943.620 & -36.739 \\
9 & 8246 & 852.768 & -31.180 & 11228 & 943.627 & -36.762 \\
\hline
\hline
  \end{tabular}
\end{table}

\subsubsection{Adaptive meshes}
Consider the~$P_4$ Lagrange element and the $\HZ_3(\omega_1^-)+\CG_4(\Omega\setminus\omega_1^-)$ method on adaptive meshes, where an initial mesh of size $h=1$~m is locally refined to $h=0.1$~m near the interface~$\Gamma_\alpha$; see Figure~\ref{fig:adaptive_mesh}.
Following~\cite{VERFURTH1999419}, adaptive meshes are generated by applying a recovery-type error estimator for the solution~$u_h$ by the $P_4$ Lagrange element  
$$
\eta_K=
    \Big\|\varepsilon(u_h)-\frac{G(u_h) +G(u_h)^T}{2}\Big\|_{0,K},~ \mbox{ for all }\,K\in\mathcal{T}_h,
$$
where $G$ is the gradient recovery operator defined in~\cite[Section 2]{Zhang2005}.

\begin{figure}[htbp]
    \centering
    % \subfloat{\includegraphics[trim=70 20 10 20, clip, width=0.25\linewidth]{adaptive_mesh_0.eps}}
    % \subfloat{\includegraphics[trim=70 20 10 20, clip, width=0.25\linewidth]{adaptive_mesh_0_fangda.eps}}
    \subfloat{\includegraphics[trim=10 20 10 20, clip, width=0.30\linewidth]{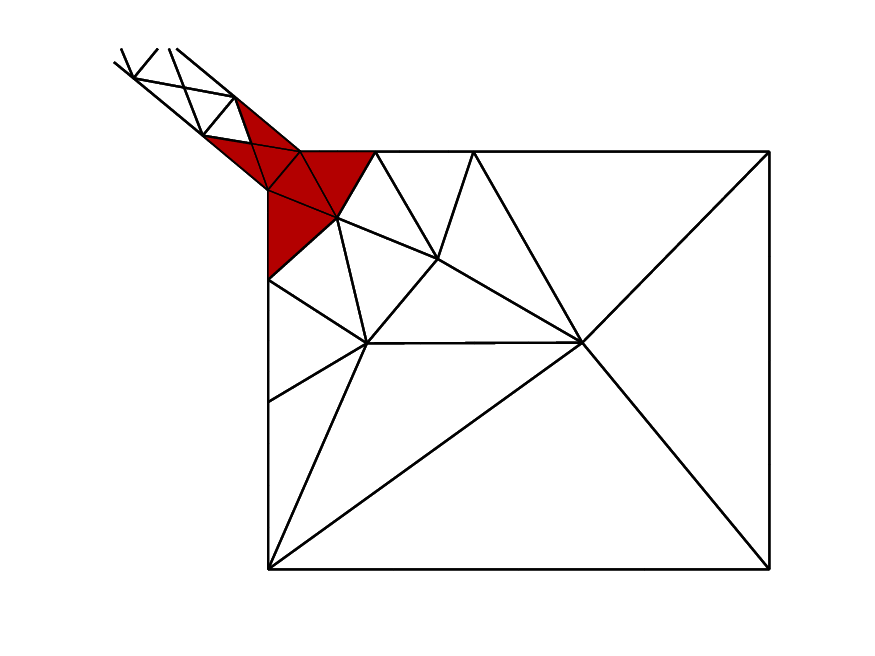}}
    \hspace{2.0em}
    \subfloat{\includegraphics[trim=10 20 10 20, clip, width=0.30\linewidth]{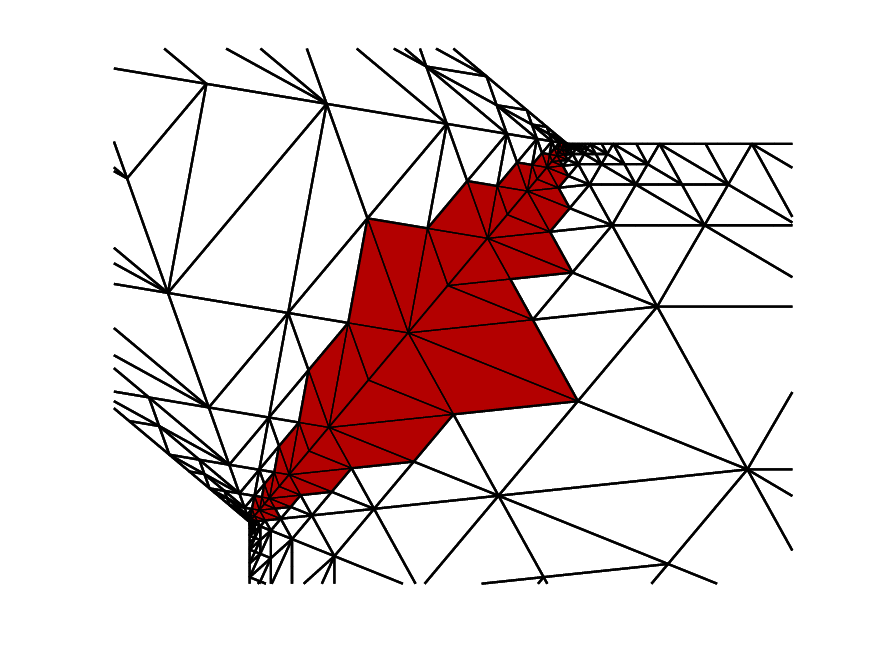}}
    \caption{The zoomed view at the junction of the initial and final refined mesh.}
    \label{fig:adaptive_mesh}
\end{figure}

Table~\ref{adaptivemeshresult} records the number of elements (NE) of each adaptive mesh, the total Dofs, shear force $Q$, and bending moment $M$ by the $P_4$ Lagrange element and the $\HZ_3(\omega_1^-)+\CG_4(\Omega\setminus\omega_1^-)$ method. The results indicate that the coupling method is stable and  convergent even on coarse meshes, whereas the $P_4$ Lagrange finite element method fails to fully converge even on the finest mesh. Remarkably, the $\HZ_3(\omega_1^-)+\CG_4(\Omega\setminus\omega_1^-)$ method achieves comparable accuracy on the initial mesh with $2344$ Dofs to that of the $P_4$ Lagrange finite element method on the final mesh with $19626$ Dofs.

Direct numerical integration results from an $hp$-Lagrange finite element method
are reported in~\cite[Table 1]{Antti2012}. Due to stress discontinuity, the computed $Q$ and $M$ differ on the shell and ring sides even with $40558$ Dofs.
Post-processing via extraction formulas~\cite[(11)--(12)]{Antti2012} significantly improves the results as shown in~\cite[Table 2]{Antti2012}.
In contrast, the continuity of the normal stress in the Hu--Zhang discretization enables direct computation of $Q$ and $M$ without post-processing, while achieving higher accuracy.

\begin{table}[htp!]
\centering
\caption{
% \centering
Numerical results of the $P_4$ Lagrange element and the coupling $\HZ_3(\omega_1^-)+\CG_4(\Omega\setminus\omega_1^-)$ method on adaptive meshes.
%ref Q=943.7(N/m), ref M=36.79(N/m).
% h = 10m, hcor = 0.01m
}
\label{adaptivemeshresult}
\begin{tabular}{c@{\hspace{0.5cm}}|c@{\hspace{0.5cm}}c@{\hspace{0.5cm}}c@{\hspace{0.5cm}}|c@{\hspace{0.5cm}}c@{\hspace{0.5cm}}c}
\hline
\hline
\multirow{2}{*}{NE}&\multicolumn{3}{c}{$P_4$ Lagrange}&\multicolumn{3}{|c}{$\HZ_3(\omega_1^-)+\CG_4(\Omega\setminus\omega_1^-)$}\\
& Dofs &Q(N/m)& M(Nm/m)&Dofs& Q(N/m)&M(Nm/m) \\
\hline
113 & 2182 & 709.813 & -22.483 & 2344 & 943.573 & -36.523 \\
%158 & 2906 & 721.285 & -18.035 & 3184 & 943.883 & -36.679 \\
167 & 3078 & 803.956 & -23.782 & 3393 & 943.650 & -36.683 \\
%182 & 3322 & 822.795 & -27.262 & 3715 & 943.690 & -36.758 \\
194 & 3530 & 840.419 & -27.924 & 4023 & 943.653 & -36.747 \\
%219 & 3934 & 850.922 & -30.363 & 4505 & 943.667 & -36.777 \\
232 & 4162 & 862.982 & -30.824 & 4833 & 943.657 & -36.772 \\
%250 & 4458 & 870.703 & -32.401 & 5207 & 943.662 & -36.784 \\
273 & 4846 & 880.017 & -32.715 & 5695 & 943.658 & -36.782 \\
%312 & 5482 & 885.351 & -33.741 & 6409 & 943.660 & -36.787 \\
326 & 5722 & 885.351 & -33.741 & 6649 & 943.660 & -36.787 \\
%359 & 6278 & 891.910 & -33.954 & 7305 & 943.659 & -36.786 \\
383 & 6678 & 900.146 & -34.668 & 7809 & 943.660 & -36.788 \\
%406 & 7066 & 913.933 & -34.799 & 8360 & 943.659 & -36.788 \\
444 & 7690 & 916.612 & -35.338 & 9140 & 943.660 & -36.789 \\
%472 & 8162 & 925.478 & -35.441 & 9812 & 943.659 & -36.789 \\
517 & 8918 & 930.962 & -35.843 & 10672 & 943.659 & -36.789 \\
%561 & 9646 & 932.379 & -35.893 & 11615 & 943.659 & -36.789 \\
612 & 10490 & 932.812 & -36.112 & 12600 & 943.659 & -36.789 \\
%689 & 11758 & 936.567 & -36.197 & 14083 & 943.659 & -36.789 \\
779 & 13238 & 938.170 & -36.334 & 15678 & 943.659 & -36.789 \\
%873 & 14790 & 937.748 & -36.393 & 17382 & 943.659 & -36.789 \\
960 & 16210 & 939.164 & -36.495 & 18917 & 943.659 & -36.789 \\
%1096 & 18474 & 942.610 & -36.545 & 21385 & 943.659 & -36.789 \\
1166 & 19626 & 943.359 & -36.586 & 22563 & 943.659 & -36.789 \\
\hline
\hline
  \end{tabular}
\end{table}

\subsection{MacNeal–Harder’s slender cantilever beams in 3D}\label{3DExam}
This example involves a benchmark problem testing the slender cantilever beam proposed by MacNeal and Harder \cite{macneal1985a}.
%The computational domain is $\Omega=\left(0,6\right)\times\left(0,0.2\right)\times\left(0,0.1\right)$. Material properties are $E=10^7$ and $\nu=0.3$.
Following \cite{macneal1985a,hu2019nonconforming}, the body force $f=(0,0,0)^T$ is placed on the domain $\Omega=\left(0,6\right)\times\left(0,0.2\right)\times\left(0,0.1\right)$ with material parameters $E=10^7$ and $\nu=0.3$. The homogeneous displacement condition $u=(0,0,0)^T$ is imposed on the left boundary $\Gamma_D=\{(x,y,z)\in\partial\Omega\,:\,x=0\}$,  the nonzero traction condition $\sigma n = (0,0,-50)^T$ on the right boundary, and zero traction on all other  boundary faces.

%\begin{itemize}
%    \setlength\itemsep{-0.00em} % 手动减小间距
%    \item Body force $f=(0,0,0)^T$,
%    \item Displacement boundary condition on the left boundary $\Gamma_D$: $u=(0,0,0)^T$,
%    \item Non-zero traction on the right boundary: $\sigma n=(0,0,-P)^T$ with $P=50$,
%    \item Zero traction on other boundaries: $\sigma n=(0,0,0)^T$.
%\end{itemize}
The first four levels of uniform tetrahedral meshes $\mathcal{T}_i$ $(1\le i\le 4)$ are obtained by partitioning~$\Omega$ into $n_x\times n_y\times n_z$ hexahedrons with $n_x=6,12,24,48$ and $n_y=n_z=1$, and each hexahedron is subdivided into six tetrahedral elements.  The initial tetrahedral mesh is displayed in Figure \ref{Fig:Beam}. According to~\cite{hu2019nonconforming}, the reference normal stress at point $B=(1,0,0.1)$ is $\sigma_{Bx}=15000.0$. In this example, let the stress concentration set be the planar interface $x=1$, and adopt the notation $\HZ_4(\omega_1^-)+\CG_1(\Omega\setminus\omega_1^-)$ as before.

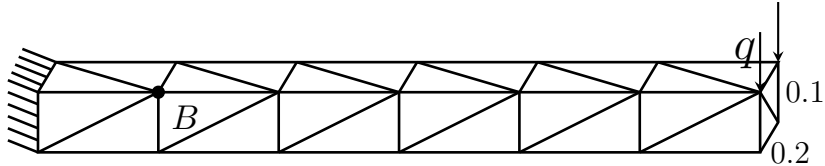
\begin{figure}[htbp]
\centering
\begin{tikzpicture}[scale=0.8, line width=1pt]
\centering
\def\len{12};
\def\he{1};
\def\a{0.3};
\def\b{0.5};
\draw (0,0) rectangle (\len,\he);
\foreach \x in {2,4,...,10} {
    \draw (\x,0) -- (\x,\he);
}

\coordinate (A) at (0+\a,\he+\b);
\coordinate (B) at (0,\he);
\coordinate (C) at (0,0);
\coordinate (D) at (\len+\a,\he+\b);
\coordinate (E) at (\len,\he);
\coordinate (F) at (\len,0);
\coordinate (G) at (\len+\a,0+\b);
\draw (A) -- (D) -- (E);
\draw (D) -- (G) -- (F);
\draw (E) -- (G);
\foreach \x in {0,2,...,10} {
    \draw (\x+\a,\he+\b) -- (\x,\he);
    \draw (\x+\a,\he+\b) -- (\x+2,\he);
    \draw (\x,0) -- (\x+2,\he);
}
\filldraw (2,1) circle (2.5pt) node[scale=1.3] [below right] {$B$};

\node at (\len+0.5,0) {\large$0.2$};
\node at (\len+0.75,\he) {\large$0.1$};

\draw[->, >=stealth, thick] (\len,\he+1) -- ++(0,-1);
\draw[->, >=stealth, thick] (\len+\a,\he+\b+1) -- ++(0,-1);
\node[scale=1.5] at (\len-0.25,\he+0.65) {\large$q$};

\foreach \i in {0,0.2,...,1} {
    \draw (-0.5,0.2+\i) -- (0,\i);
}

\draw (-0.26,1.72) -- (0.3,1.5);
\draw (-0.32,1.59) -- (0.2250,1.375);
\draw (-0.38,1.46) -- (0.15,1.25);
\draw (-0.44,1.33) -- (0.075,1.1125);

\end{tikzpicture}
\caption{The initial tetrahedral mesh for the cantilever beam with $n_x=6$, $n_y=n_z=1$.}
\label{Fig:Beam}
\end{figure}

% This example computes the tip deflection $u_A$ in the direction of $(0,0,-1)^T$ at point $A=(6,0.1,0.05)^T$, the energy $\Pi$ and the normal stress $\sigma_{Bx}$ at point $B=(1,0,0.1)^T$. The reference values $u_A=0.4321$ and $\Pi=-0.2159$ are given in \cite{macneal1985a}. In this paper, the references $u_A=0.4310$, $\Pi=-0.2155$, and $\sigma_{Bx}=15000.0$ are computed using the $P_5$ Lagrange element on the fourth level mesh, consistent with \cite{hu2019nonconforming}.

% Table \ref{canBeams} presents results obtained with the $P_1$ Lagrange element, the $P_4$ Hu--Zhang element, and the coupling $P_4$ Hu--Zhang and $P_1$ Lagrange element. In the coupling method, the $P_4$ Hu--Zhang element is implemented within the subdomain $\Omt=\{(x,y,z)^T\in\mathbb{R}^3: 1-h_x<x<1+h_x\}$ with $h_x=\frac{6}{n_x}$. Compared with the $P_1$ Lagrange element, the coupling method achieves significant improvement in $\sigma_{Bx}$ accuracy while enhancing in $u_{A}$ and $\Pi$ accuracies, despite adding only 2000 Dofs.
% Although the $P_4$ Hu--Zhang element achieves higher accuracy for both $u_A$ and $\Pi$, its Dofs substantially exceed those of the coupling method. Notably, on the last two mesh levels, the coupling method uses only 10\% of the Dofs required by the $P_4$ Hu--Zhang element while maintaining a relative error in $\sigma_{Bx}$ below 1\%.

Table~\ref{canBeams} records the numerical values of $\sigma_{Bx}$ by the $P_1$ Lagrange element, the $P_4$ Hu--Zhang element, and the $\HZ_4(\omega_1^-)+\CG_1(\Omega\setminus\omega_1^-)$ method, where $\mathcal{T}_i$ with $i\ge 5$ is generated by uniform refinement of $\mathcal{T}_{i-1}$.
On $\mathcal{T}_3$–$\mathcal{T}_4$, both the $P_4$ Hu--Zhang and coupling methods achieve errors below $1\%$, while the coupling method uses only about $10\%$ of the DoFs. It also outperforms the $P_1$ Lagrange method at comparable DoFs and remains more accurate even when the latter uses $667{,}590$ DoFs versus $2{,}085$ DoFs.

\begin{table}[htbp]
		\centering \caption{Numerical results of the $P_1$ Lagrange element , the $P_4$ Hu--Zhang element and the  $\HZ_4(\omega_1^-)+\CG_1(\Omega\setminus\omega_1^-)$ method for Cantilever beams.}\label{canBeams}
	\begin{tabular}{c|ccccc}
		\hline
		\multicolumn{1}{c|}{Mesh}                                                              & \multicolumn{1}{c}{} & \multicolumn{1}{c}{$\mathcal{T}_1$} & \multicolumn{1}{c}{$\mathcal{T}_2$} & \multicolumn{1}{c}{$\mathcal{T}_3$} & \multicolumn{1}{c}{$\mathcal{T}_4$} \\ \hline
		\multirow{2}{*}{\begin{tabular}[c]{@{}c@{}}$P_1$ Lagrange\end{tabular}} & Dofs                 & 96                                  & 168                   & 312                   & 600                   \\
		& $\bm{\sigma_{Bx}}$   & 19.3                                & 83.2                  & 276.0                 & 594.8                 \\ \hline
		\multirow{2}{*}{$P_4$ Hu--Zhang}                                                              & Dofs                 & 5682                                & 11262                 & 22422                 & 44742                 \\
		& $\bm{\sigma_{Bx}}$   & 14826.1                             & 14985.4               & 14999.2               & 15000.3               \\ \hline
		\multirow{2}{*}{$\HZ_4(\omega_1^-)+\CG_1(\Omega\setminus\omega_1^-)$}                                                                & Dofs                 & 2085                                & 2181                  & 2325                  & 2613                  \\
		& $\bm{\sigma_{Bx}}$   & 14833.4                             & 14765.4               & 15031.0               & 15052.6               \\ \hline
		\multicolumn{1}{c|}{Mesh}                                                             & \multicolumn{1}{c}{} & \multicolumn{1}{c}{$\mathcal{T}_5$}               & \multicolumn{1}{c}{$\mathcal{T}_6$} & \multicolumn{1}{c}{$\mathcal{T}_7$} & \multicolumn{1}{c}{$\mathcal{T}_8$} \\ \hline
		\multirow{2}{*}{\begin{tabular}[c]{@{}c@{}}$P_1$ Lagrange\end{tabular}}  & Dofs                 & 2646                                & 14550                 & 93798                 & 667590                \\
		& $\bm{\sigma_{Bx}}$   & 3656.6                              & 8702.2                & 12415.9               & 14023.1               \\ \hline
	\end{tabular}
\end{table}

\bibliographystyle{siamplain}
\bibliography{ref}

\end{document}